\documentclass[smallextended]{svjour3}

\smartqed

\input xy

\xyoption{all}

\usepackage{latexsym,amssymb,amsbsy,amsmath,graphicx,epsfig,times}

\newtheorem{thm}{Theorem}[section]
\newtheorem{lem}[thm]{Lemma}
\newtheorem{prop}[thm]{Proposition}
\newtheorem{cor}[thm]{Corollary}
\newtheorem{dfn}[thm]{Definition}
\newtheorem{ques}[thm]{Question}

\renewcommand{\bf}[1]{\mathbf{#1}}
\renewcommand{\rm}[1]{\mathrm{#1}}
\renewcommand{\cal}[1]{\mathcal{#1}}

\newcommand{\bbC}{\mathbb{C}}

\newcommand{\bbN}{\mathbb{N}}

\newcommand{\bbR}{\mathbb{R}}
\newcommand{\bbT}{\mathbb{T}}
\newcommand{\bbZ}{\mathbb{Z}}

\newcommand{\sfP}{\mathsf{P}}

\newcommand{\rmH}{\mathrm{H}}

\renewcommand{\d}{\mathrm{d}}

\newcommand{\m}{\mathrm{m}}

\newcommand{\A}{\mathcal{A}}
\newcommand{\B}{\mathcal{B}}
\newcommand{\C}{\mathcal{C}}

\newcommand{\E}{\mathcal{E}}
\newcommand{\F}{\mathcal{F}}

\newcommand{\Z}{\mathcal{Z}}

\newcommand{\frH}{\mathfrak{H}}

\newcommand{\G}{\Gamma}
\renewcommand{\O}{\Omega}
\renewcommand{\S}{\Sigma}

\renewcommand{\a}{\alpha}
\renewcommand{\b}{\beta}
\newcommand{\eps}{\varepsilon}
\newcommand{\g}{\gamma}
\renewcommand{\k}{\kappa}
\renewcommand{\l}{\lambda}

\newcommand{\s}{\sigma}

\newcommand{\img}{\mathrm{im}\,}

\renewcommand{\hat}[1]{\widehat{#1}}
\newcommand{\ol}[1]{\overline{#1}}

\newcommand{\into}{\hookrightarrow}
\newcommand{\onto}{\twoheadrightarrow}
\newcommand{\fin}{\nolinebreak\hspace{\stretch{1}}$\lhd$}

\renewcommand{\t}[1]{\tilde{#1}}
\newcommand{\actson}{\curvearrowright}
\renewcommand{\to}{\longrightarrow}

\begin{document}

\title{Continuity properties of measurable group cohomology}
\author{Tim Austin\thanks{TA's research supported by fellowships from Microsoft Corporation and from the Clay Mathematics Institute} \and Calvin C. Moore}

\titlerunning{Moore cohomology}

\institute{\textsc{Tim Austin}\at
              Department of Mathematics,\\ Brown University, Providence, RI 02912, USA,\\ \email{timaustin@math.brown.edu} \and \textsc{Calvin C. Moore}\at Department of Mathematics,\\ University of California, Berkeley, CA 94720, USA,\\ \email{ccmoore@math.berkeley.edu}}

\date{}

\maketitle

\begin{abstract}
A version of group cohomology for locally compact groups and Polish modules has previously been developed using a bar resolution restricted to measurable cochains.  That theory was shown to enjoy analogs of most of the standard algebraic properties of group cohomology, but various analytic features of those cohomology groups were only partially understood.

This paper re-examines some of those issues.  At its heart is a simple dimension-shifting argument which enables one to `regularize' measurable cocycles, leading to some simplifications in the description of the cohomology groups.  A range of consequences are then derived from this argument.

First, we prove that for target modules that are Fr\'echet spaces, the cohomology groups agree with those defined using continuous cocycles, and hence they vanish in positive degrees when the acting group is compact.  Using this, we then show that for Fr\'echet, discrete or toral modules the cohomology groups are continuous under forming inverse limits of compact base groups, and also under forming direct limits of discrete target modules.

Lastly, these results together enable us to establish various circumstances under which the measurable-cochains cohomology groups coincide with others defined using sheaves on a semi-simplicial space associated to the underlying group, or sheaves on a classifying space for that group.  We also prove in some cases that the natural quotient topologies on the measurable-cochains cohomology groups are Hausdorff.
\end{abstract}

\tableofcontents

\section{Introduction}

\subsection{Cohomology for locally compact groups}

The cohomology of discrete groups came into being in the works~\cite{EilMac47-1,EilMac47-2} of Eilenberg and MacLane.  It emerged from Hurewicz' classical discovery that the cohomology groups of two aspherical simplicial complexes are equal if those complexes have isomorphic fundamental groups~\cite{Hur36}, and it then quickly developed in papers of Eilenberg, MacLane, Hopf, Eckmann and others.

In addition to clarifying the structure of these invariants from algebraic topology, this theory has proved useful within group theory in various ways. On the one hand, the low-degree (degree 1 and 2) cohomology classes were found to correspond with naturally-defined data describing crossed homomorphisms and Abelian extensions of groups, and so enabled a streamlined understanding of those data. On the other, cohomology classes in various degrees can be associated to a wide range of different kinds of action of a group, and so can help in understanding those actions.  In particular, the theory often allows the issue of whether a given action admits some additional structure to be boiled down to its simplest possible residue in the form of a functional equation which may or may not have a solution.  This ability to treat obstructions systematically has led to further points of interaction with both algebraic topology and number theory.

These aspects of the cohomology of abstract groups are by now described in a range of standard texts, such as Brown's thorough and accessible book~\cite{Bro82}.  Since the 1950s group cohomology has also taken its place as a central motivating example within the more abstract study of homological algebra, for which we recommend Weibel's treatment~\cite{Wei94}. In addition, introductions to various instances of interplay between group theory and topology that result from cohomology theory can be found in Bredon~\cite{Bre67} and Thomas~\cite{Tho86}.

Almost immediately, the question poses itself of how to define a similar theory for topological groups $G$ and topological $G$-modules $A$, and to reap the same benefits for topological groups as in the case of discrete groups.

We consider an Abelian topological group $A$ upon which $G$ acts as a topological transformation group of automorphisms. One issue is what kinds of topological group to consider. The most natural choice for $G$ is a locally compact group, and we will restrict ourselves here to that class, and in addition will always assume that any locally compact group discussed satisfies the second axiom of countability.  At first sight it is natural to assume that $G$-modules $A$ should also be locally compact, but it is essential both for applications and for the coherence of the theory that this class be expanded to the class of Polish $G$-modules (a Polish topological group is one that admits a complete and separable metric). In addition to all locally compact $G$-modules, this class includes all separable Banach and Hilbert spaces, plus a significant class of other separable topological vector spaces (the F-spaces).  It also includes important function spaces such as the space of all measurable functions from a standard Borel $\s$-finite measure space $(Y,\S,\mu)$ into a Polish group $A$, where this is given the topology of convergence in measure. Also, a countable product of Polish groups is Polish.

We denote the category of Polish $G$-modules by $\sfP(G)$.  Exact sequences in $\sfP(G)$ are sequences
\begin{eqnarray}\label{eq:sh-ex}
(0) \to A \stackrel{i}{\to} B \stackrel{j}{\to} C \to (0)
\end{eqnarray}
which are exact algebraically with $i$ and $j$ continuous. It follows from standard results on Polish groups~\cite{Ban55} that $i$ is automatically a homeomorphism onto its range and that $j$ automatically induces a homeomorphism of the quotient group $B/i(A)$ with $C$, so that the sequence is exact in a very strong sense.

What is sought in this context is a family of covariant functors $\rmH^n(G,\cdot)$, $n\geq 0$, from $\sfP(G)$ to Abelian groups with $\rmH^0(G,A) = A^G$, the subgroup of $G$-fixed points in $A$, and so that to every exact sequence~(\ref{eq:sh-ex}) in $\sfP(G)$ there are dimension-shifting connecting homomorphisms $\rmH^n(G,C) \to \rmH^{n+1}(G,A)$ (sometimes called `switchbacks') so that everything fits together to form an infinite exact sequence of cohomology
\[ (0) \to A^G \to B^G \to C^G \to \rmH^1(G,A) \to \rmH^1(G,B) \to \ldots..\]
Call such a family of functors, a cohomological functor $\rmH^\ast(G,\cdot)$. Finally, we want these functors to be effaceable in $\sfP(G)$, which means that for any $A$ and any $a \in \rmH^n(G,A)$ there is an exact sequence~(\ref{eq:sh-ex}) so that the image of $a$ in $\rmH^n(G,B)$ vanishes.  By a well-known argument of Buchsbaum~\cite{Buc60}, an effaceable cohomological functor on $\sfP(G)$, if one can be found, must be unique. Importantly, unlike discrete group cohomology, such an effaceable cohomological functor cannot be obtained by computing derived functors from injective resolutions, because the category $\sfP(G)$ generally does not have enough injectives.

One of us introduced in a series of papers beginning in 1964~\cite{Moo64(gr-cohomI-II),Moo76(gr-cohomIII),Moo76(gr-cohomIV)} cohomology groups satisfying these requirements, based on a complex of measurable cocycles from $G^n$ into $A$ (or more properly equivalence classes of measurable functions which agree almost everywhere) with the usual coboundary operator. (Precise definitions will be recalled in Section~\ref{sec:prelim} below.) This leads to cohomology groups we denote by $\rmH_\m^\ast(G,A)$ (`m' standing for `measurable'). One could also use Borel functions, but this leads to the same end result. These groups have the right values in degree $0$, and the dimension-shifting connecting maps always exist and fit into a long exact sequence of cohomology corresponding to any short exact sequence~(\ref{eq:sh-ex}). Finally they are effaceable, and hence form the unique effaceable cohomological functor on $\sfP(G)$.

In addition, the groups $\rmH_\m^n(G,A)$ have a natural topology which comes from the quotient structure
$\Z^n(G,A)/\B^n(G,A)$ where $\Z^n(G,A)$, the group of measurable cocycles, is itself Polish and where $\B^n(G,A)$ is the continuous image of another Polish group.  If the latter group happens to be closed in the former group, or equivalently if the quotient is Hausdorff, then $\rmH_\m^n(G,A)$ is also a Polish group. However, the coboundaries are not always closed and in this case $\rmH_\m^n(G,A)$ still has a topology, but it is not particularly useful.

Although $\rmH^\ast_\m$ is the unique effaceable cohomological functor on $\sfP(G)$, many different candidates for a cohomology theory of topological groups have emerged over the years based on other sets of requirements. We will now briefly discuss five of these candidates.

The first appearance of cohomology groups of topological groups was in class field theory, where Galois groups of infinite field extensions and their cohomology groups appeared naturally. These Galois groups are profinite and so compact and totally disconnected, and the cochains that appeared naturally were continuous ones (see~\cite{ArtTat51} and also~\cite{ArtTat09}). This  cohomology theory for infinite Galois groups and its applications to class field theory were developed by Tate and his students. At almost the same time van Est began a study of the cohomology of Lie groups acting on finite dimensional vector spaces, again using a complex of continuous cochains~\cite{vEs53}.  This work was continued and extended by Hochschild and Mostow~\cite{Mos61,Hoc61,HocMos62} and then developed further by Borel and Wallach~\cite{BorWal00}, who specifically included general infinite-dimensional Fr\'echet  spaces as $G$-modules.

In general, for any $G$ and any Polish $G$-module $A$, one can introduce the cochain complex of continuous functions from $G^n$ to $A$ with the usual coboundary operator on cochains. This is the most natural and straightforward generalization from the case of discrete groups, and produces a family of functors that we denote $\rmH^\ast_\rm{cts}(G,\cdot)$.  They satisfy $\rmH^0_\rm{cts}(G,A) = A^G$ and are effaceable.  In addition, if $G$ is totally disconnected and $A$ is arbitrary, or if $G$ is arbitrary and $A$ is restricted to the subcategory of modules that are Fr\'echet spaces, then the dimension-shifting connecting maps do exist and there is a long exact sequence of cohomology.  In both cases this is a consequence of results giving continuous lifts for short exact sequences of modules under these assumptions\footnote{In fact, earlier works such as Hochschild and Mostow's~\cite{HocMos62} simply narrowed the requirement for long exact sequences to only those short exact sequences of modules that admit continuous left-inverses as sequences of topological spaces.  With this convention, the theory $\rmH^\ast_\rm{cts}$ always has long exact sequences, but on the other hand the theory is not effaceable in $\sfP(G)$ if one allows only inclusions of modules that give rise to such distinguished short exact sequences.}. These are the two cases where these groups were introduced and developed very successfully. However, for general $G$ and $A$ and short exact sequences as in~(\ref{eq:sh-ex}) the dimension-shifting connecting maps for $\rmH^\ast_\rm{cts}$ are missing, so there is not always a long exact sequence of cohomology.  We note that, as for $\rmH^n_\m(G,A)$, the groups $\rmH_\rm{cts}^n(G,A)$ inherit a topology as a quotient of the $n$-cocycles (with the compact-open topology) modulo the $n$-coboundaries, which also may or may not be Hausdorff.

In 1973 Wigner~\cite{Wig73} introduced two further cohomology theories.  The first is based on equivalence classes of multiple extensions of elements of $\sfP(G)$, as studied by Yoneda~\cite{Yon60} in the non-topological case. The multiple extensions in dimension $n$ are exact sequences of the form
\[ (0) \to A \to E_1 \to E_2 \to \ldots \to E_n \to \bbZ \to (0)\]
where the $E_i$ are in $\sfP(G)$ and where $\bbZ$ is given the trivial $G$-action. This construction leads to cohomology groups $\rmH^\ast_\rm{YW}(G,A)$ (`$\rm{YW}$' for `Yoneda-Wigner') which are easily seen to be effaceable cohomological functors with the correct value for $\rmH^0$. Hence by the uniqueness theorem, Wigner concludes that $\rmH_\m^\ast(G,A) = \rmH^\ast_\rm{YW}(G,A)$ for all $G$ and $A$; we will not discuss $\rmH^\ast_\rm{YW}$ further in this paper.

Secondly, Wigner adapts a construction of Grothendieck, Artin, Verdier, and Deligne to this topological context.  Wigner builds a semi-simplicial space $G^\bullet$ over $G$, which at level $n$ is $G^n$, together with a semi-simplicial sheaf $\cal{A}^\bullet$ corresponding to any $G$-module $A$ in $\sfP(G)$. The sheaf at level $n$ consists of germs of continuous functions from $G^n$ into $A$. One then resolves this semi-simplicial sheaf to get cohomology groups $\rmH^\ast_\rm{ss}(G,A)$ (`$\rm{ss}$' for `semi-simplicial'). Wigner shows these are effaceable functors with the right value $A^G$ for $n=0$. However, the dimension-shifting connecting maps do not always exist. It is easy to see that they do exist if the exact sequence of $G$-modules~(\ref{eq:sh-ex}) has continuous local cross-sections from $C$ back up to $B$, but without this assumption a short exact sequence of modules may not correspond to a short exact sequence of sheaves on $G^\bullet$.  One can construct long exact sequences in $\rmH^\ast_\rm{ss}$ for short exact sequences of sheaves, but one has an adequate supply of short exact sequences only after enlarging the class of permitted sheaves on $G^\bullet$ to examples not arising from members of $\sfP(G)$.

Wigner's main result here is that if $G$ is finite-dimensional and if $A$ is a member of a certain subcategory $\sfP_\rm{F}(G)$ of $\sfP(G)$, consisting of Polish modules that have something he calls `Property F', then connecting maps always exist using only the sheaves constructed from objects of $\sfP_\rm{F}(G)$; and by the Buchsbaum uniqueness theorem applied to the category $\sfP_\rm{F}(G)$, it follows that $\rmH^\ast_\rm{ss}(G,A) = \rmH_\m^\ast(G,A)$ for $G$ finite-dimensional and $A$ in $\sfP_\rm{F}(G)$. (One also has to check that both of these theories are effaceable within the subcategory, which is the case.)  Later we will give an example of $G$ and $A$ where $\rmH^\ast_\rm{ss}(G,A)$ is not the same as $\rmH_\m^\ast(G,A)$, so Wigner's restrictions are not superfluous. We note that the cohomology theory $\rmH^\ast_\rm{ss}$ has been further developed and refined in recent papers of Lichtenbaum~\cite{Lic09} and Flach~\cite{Fla08}, working via a more general topos-theoretic redefinition based on ideas of Grothendieck, Artin and Verdier~\cite{SGA4t1,SGA4t2,SGA4t3} which always gives the same theory as Wigner's in his setting.  Those more recent works are motivated by number-theoretic applications to the interpretation of values of the Dedekind Zeta-functions of number fields.

Graeme Segal in 1970~\cite{Seg70} developed another cohomology theory of topological groups using a type of contractible resolution of a $G$-module $A$. To be comparable, $G$ and $A$ have to be k-spaces, which they are in our case; all $G$-modules have to be locally contractible; and only those exact sequences of $G$-modules which have topological local cross-sections are allowed. Unlike the category of Polish $G$-modules, this category admits a natural family of canonical resolutions of such modules, and using these Segal defines his theory much as the cohomology of classical groups can be obtained as a sequence of derived functors. We call this theory $\rmH^\ast_\rm{Seg}(G,A)$, and it is clear that in the context where they are both defined one has $\rmH^\ast_\rm{ss}(G,A) = \rmH^\ast_\rm{Seg}(G,A)$ (see~\cite{Seg70} Section 3).

The last cohomology theory for topological groups we mention is a direct generalization of the topological interpretation of cohomology for discrete groups in terms of classifying spaces. For a topological group $G$ one may construct a locally trivial principal $G$-bundle $E_G \to B_G$ with $E_G$ contractible, universally up to homotopy, and having done so the base $B_G$ is called a classifying space for $G$.  See, for instance, Husem\"oller~\cite{Hus95}, and also the monograph~\cite{HofMos73} of Hofmann and Mostert.  In our case these spaces may be taken to be paracompact. Then for any $A$ in $\sfP(G)$, one can form a locally trivial associated fibre bundle over $B_G$ with fibre $A$, and then form the sheaf $\cal{A}$ of germs of continuous sections of this bundle; see~\cite{Seg70} Proposition 3.3.  The sheaf cohomology groups $\rmH^\ast(B_G,\cal{A})$ are possible candidates for cohomology groups of $G$. However there is a problem in that $B_G$ is defined only up to homotopy type and sheaf cohomology may not be homotopy-invariant unless the sheaf is locally constant. This means that for a well-defined theory we need to restrict to discrete $A$, so that $\cal{A}$ is locally constant. (Although having fixed a choice of $B_G$ the groups $\rmH^\ast(B_G,\cal{A})$ for more general $A$ do have an auxiliary r\^ole in Segal's work: he uses dimension-shifting to construct a comparison $\rmH^\ast_\rm{Seg}(G,A)\to \rmH^\ast(B_G,\cal{A})$, and so he cannot stay within the class of discrete modules.) For a discrete $G$-module $A$, we denote
$\rmH^\ast_\rm{cs}(G,A) := \rmH^\ast(B_G,\cal{A})$ (`$\rm{cs}$' for `classifying space').  For discrete $G$ these coincide with the usual cohomology for discrete groups, since in this case $B_G$ is a $\rm{K}(G,1)$. Segal~\cite{Seg70} shows that $\rmH^\ast_\rm{Seg}$ and $\rmH^\ast_\rm{cs}$ agree when both are defined, that is for discrete $A$. Wigner also proves that $\rmH^\ast_\rm{ss}$ and $\rmH^\ast_\rm{cs}$ agree in this case.

While all of these theories play important r\^oles, clearly $\rmH^\ast_\m$ and $\rmH^\ast_\rm{ss}$ have the widest scope in terms of permissible coefficients $A$ and groups $G$, and
our primary focus here will be on $\rmH^\ast_\m$ and on expanding the area of its agreement with $\rmH^\ast_\rm{ss}$. Various analytic and computational questions about these cohomology groups remained open following~\cite{Moo64(gr-cohomI-II),Moo76(gr-cohomIII),Moo76(gr-cohomIV)}, and the present paper resolves some of these.  Our principal results are of three kinds:
\begin{itemize}
\item sufficient conditions for the cohomology bifunctors to be continuous under inverse or direct limits of the arguments;
\item sufficient conditions for the cohomology groups to be Hausdorff in their quotient topologies;
\item and conditions under which $\rmH_\m^\ast$ can be shown to agree with another of the theories discussed above.
\end{itemize}
As we shall see, the proofs of these results have certain key analytic arguments in common.

Before proceeding further, let us offer another general note concerning measurable cohomology $\rmH_\m^\ast(G,A)$.  It is reasonable to ask why this cohomology theory for topological groups $G$ and $A$ based on measurable cochains, which themselves appear to be so very weakly linked to the topological structure of $G$ and $A$, should work at all. Two points should be made. First, the topology on a locally compact group can be constructed or reconstructed from the measure theory --- that is, from the Borel structure and an invariant or quasi-invariant measure on $G$. Mackey established this in~\cite{Mac57} by refining an earlier result of Weil (which first appeared in~\cite{Wei64}). Hence the measure theory on $G$ determines the topology on $G$ in a rather strong sense. Second, one has a variety of automatic continuity theorems for Polish groups such as the following.  If $A$ is Polish group and $f:A\to B$ is a homomorphism into a separable metric group $B$ which is a Borel function (in that the inverse image of every Borel set in $B$ is a Borel set in $A$), then $f$ is continuous (p.23 in~\cite{Ban55}). So again for Polish groups $A$, the Borel structure has a lot to say about the topology of $A$.

\subsection{Comparison with continuous cochains}

A major concern of this paper will be to expand the areas of agreement between $\rmH_\m^\ast(G,A)$, $\rmH^\ast_\rm{cts}(G,A)$, $\rmH^\ast_\rm{ss}(G,A)$, and $\rmH^\ast_\rm{cs}(G,A)$.  The first two of these are defined simply in terms of cochains (measurable or continuous, respectively) in a bar resolution, and so their comparison will use only more elementary arguments.

Simple examples show that these theories can differ, and the continuous-cochains theory lacks both the universality properties of $\rmH_\m^\ast$ when considered on the whole of $\sfP(G)$ and also the correct interpretation in terms of group extensions in degree $2$ (rather, it captures those group extensions that admit global continuous sections).  However, it has long been suspected that these theories do coincide in the following situation.

\vspace{7pt}

\noindent\textbf{Theorem A}\quad\emph{If $G$ is a locally compact second countable group and $A$ is a Fr\'echet $G$-module then the natural comparison homomorphism
\[\rmH^\ast_{\rm{cts}}(G,A)\to \rmH_\m^\ast(G,A)\]
from continuous cohomology is an isomorphism: that is, every class in $\rmH_\m^\ast(G,A)$ has a continuous representative in the bar resolution, and if the class is trivial then that continuous representative is the coboundary of another continuous cochain.}

\emph{If in addition $G$ is compact then $\rmH_\m^p(G,A) = (0)$ for all $p > 0$.}

\vspace{7pt}

Note that the analogous agreement $\rmH^\ast_{\rm{cts}}(G,A)\cong \rmH_\rm{ss}^\ast(G,A)$ is given by Theorem 3 of~\cite{Wig73}.

If $A$ is a Fr\'echet module and $G$ is compact, then given a \emph{bounded} measurable cocycle $G^n \to A$ a standard averaging trick shows that it is a coboundary.  We prove Theorem A below using the same idea, but before it can be applied one must show that all cohomology classes have measurable-cocycle representatives that are `locally bounded' in a suitable sense. This follows from an elementary procedure for regularizing measurable cocycles based on dimension-shifting, which is the first innovation of the present paper.  A similar reduction to bounded cocycles already appeared in the proof of Theorem 2.3(1) of~\cite{Moo64(gr-cohomI-II)} concerning $\rmH_\m^2(\cdot,\cdot)$, but that relied instead on the interpretation of this group in terms of group extensions together with deep results of Mackey and Weil on the structure of standard Borel groups.  By contrast, the proof of Theorem A below uses only elementary analysis of cocycles themselves.

\subsection{Continuity properties and inverse and direct limits}

Another aspect of the groups $\rmH_\m^\ast$ that has remained unclear is their behaviour under forming inverse limits in the first argument or direct limits in the second.  If $\pi:G' \onto G$ is a continuous epimorphism of locally compact second countable groups and $A$ is a Polish $G$-module (which we identify also as a $G'$-module by composing with the epimorphism), then the natural inflation maps are a sequence of homomorphisms $\rm{inf}^p_\pi:\rmH_\m^p(G,A)\to \rmH_\m^p(G',A)$.  As a result, if $(G_m)_m$, $(\pi^m_k)_{m\geq k}$ is an inverse system of locally compact second countable groups with a locally compact second countable inverse limit $G$, $(\pi_m)_m$, and $A$ is a module for all these groups (say, lifted from a module for some minimal group in the system), we may form the direct limit of the associated inflation maps to obtain a chain of homomorphisms
\[\lim_{m\rightarrow}\rm{inf}^p_{\pi_m}:\lim_{m\rightarrow}\rmH_\m^p(G_m,A)\to \rmH_\m^p(G,A).\]

A natural question is whether these are isomorphisms.  It is clear that this is not always so: for example, if $(G_m)_{m\geq 1}$ is a sequence of finite-dimensional compact Abelian groups converging to an infinite dimensional compact Abelian group $G$ and we set $A := G$ with the trivial action of each $G_m$, then the identity map $G\to A$ is a Borel crossed homomorphism, hence a $1$-cocycle, but it clearly is not lifted from any of the groups $G_m$; and in this case there are no $1$-coboundaries, so the cohomology class of this identity element contains no other cocycles, and hence that class too is not lifted from any $G_m$.

However, in this paper we show that for compact base groups $G_m$ and for certain large classes of target module the above continuity under inverse limits does in fact obtain.

\vspace{7pt}

\noindent\textbf{Theorem B}\quad\emph{
If $(G_m)_m$, $(\pi^m_k)_{m\geq k}$ is an inverse system of compact groups with second countable inverse limit $G$, $(\pi_m)_m$ then
\[\rmH_\m^p(G,A) \cong \lim_{m\rightarrow}\rmH_\m^p(G_m,A)\]
under the direct limit of the inflation maps $\rm{inf}^p_{\pi_m}:\rmH_\m^p(G_m,A)\to \rmH_\m^p(G,A)$ whenever $A$ is a discrete Abelian group or a finite-dimensional torus with an action of $G$ that factorizes through every $\pi_m$.}

\vspace{7pt}

\noindent\emph{Remark}\quad Karl Hofmann has pointed out to us that there are many naturally-occurring functors of topological algebra that respect inverse limits of compact groups, as above, but do not respect more general categorial limits.  It is quite possible this situation holds for $\rmH_\m^\ast$, but we have not examined this question.  These issues are discussed very generally in~\cite{Hof76}. \fin

\vspace{7pt}

Since a torus is a quotient of a Euclidean space by a lattice, cohomology for discrete and toral targets can be connected using the long exact sequence of cohomology and an appeal to Theorem A.  The heart of Theorem B is therefore the case of discrete $A$, and this relies on a more quantitative version of our basic cocycle-smoothing procedure.

Similarly to the above, if now $G$ is fixed, $A$ is a discrete $G$-module and $A_m \subseteq A$ is some direct system of discrete submodules with union the whole of $A$, then these inclusions define a system of homomorphisms
\[\lim_{m\rightarrow}\rm{inc}_{\iota_m}^p:\lim_{m\rightarrow}\rmH_\m^p(G,A_m)\to \rmH_\m^p(G,A),\]
and once again we may ask whether these are isomorphisms.  In fact the same basic estimates as needed for Theorem A show that this is often also the case.

\vspace{7pt}

\noindent\textbf{Theorem C}\quad\emph{
If $G$ is any compact second countable group, $A$ a discrete $G$-module and $A_m\subseteq A$ a direct system of submodules such that $\bigcup_m A_m = A$ then
\[\rmH_\m^p(G,A) \cong \lim_{m\rightarrow}\rmH_\m^p(G,A_m)\]
under the direct limit of the maps on cohomology induced by the inclusions $A_m \subseteq A$.}

\vspace{7pt}

\noindent\emph{Remark}\quad Simply by applying Theorem C and then Theorem B in turn, one can deduce at once a strengthening of Theorem B to the case when $A$ is an arbitrary discrete Abelian $G$-module, according to which
\[\rmH_\m^p(G,A) \cong \lim_{m\rightarrow}\rmH_\m^p(G_m,A^{\ker \pi_m})\]
where $A^{\ker\pi_m}$ denotes the submodule of elements of $A$ that are individually fixed by the kernel of $\pi_m:G\onto G_m$, which may be re-interpreted as a $G_m$-module.  The applicability of Theorem C to this situation follows because the orbits of the $G$-action on $A$ must be compact, hence finite, and so every element of $A$ is fixed by some $\ker \pi_m$: that is, $A = \bigcup_m A^{\ker\pi_m}$. \fin

\vspace{7pt}

A few low-degree cases of the above results already appear as Theorems 2.1, 2.2 and 2.3 in Part I of~\cite{Moo64(gr-cohomI-II)}, also proved using the group-theoretic identification of the elements of $\rmH_\m^2(G,A)$ with locally compact extensions of $A$ by $G$.

For non-compact base groups, both Theorems B and C can fail. We will exhibit a finitely-generated discrete group $\G$ such that surjectivity fails in Theorem B for some increasing sequence of discrete $\G$-modules, and will then use this to show that surjectivity also fails in Theorem A for the inverse sequence of quotients $\G\times \bbT^\infty\to \G\times \bbT^N$, $N\geq 1$, and with target module equal to $\bbT$.  We suspect that injectivity can also fail in both cases, but have not constructed examples to show this.

\subsection{Topologies on cohomology groups and further comparison results}

The remainder of our work concerns two other sets of questions about $\rmH_\m^\ast$, which turn out to be closely intertwined.

First, each group $\rmH_\m^p(G,A)$ inherits a quotient topology as the quotient of the group of measurable cocycles by that of coboundaries, and it is natural to ask when this topology is Hausdorff.  This question has some intrinsic interest, but it also has consequences for the algebraic properties of the theory: an important part of~\cite{Moo76(gr-cohomIII)} is the development of a Lyndon-Hochschild-Serre spectral sequence for calculating the cohomology of a group extension, but this makes sense only when the topologies on the groups $\rmH_\m^p(\cdot,\cdot)$ for the kernel of the extension are Hausdorff.  The remarks following Lemma I.1.1 in~\cite{Moo64(gr-cohomI-II)} offer some further discussion of this issue, and it also arises in Chapter IX of Borel and Wallach~\cite{BorWal00} (particularly Section IX.3) and Remark 12.0.6 in Monod~\cite{Mon01} for related theories that are specific to Fr\'echet target modules.

Some special cases in which the quotient topology is Hausdorff are given in~\cite{Moo64(gr-cohomI-II),Moo76(gr-cohomIII),Moo76(gr-cohomIV)}. The following result is rather more general than those.

\vspace{7pt}

\noindent\textbf{Theorem D}\quad \emph{If $G$ is almost connected (that is, its identity component $G_0$ is co-compact) then the cohomology groups $\rmH_\m^p(G,A)$ are Hausdorff in their quotient topologies in all of the following cases:}
\begin{itemize}
\item \emph{$A$ a Euclidean space, in which case each $\rmH_\m^p(G,A)$ is also Euclidean in its quotient topology;}
\item \emph{$A$ discrete, in which case each $\rmH_\m^p(G,A)$ is discrete and countable;}
\item \emph{$A$ a torus, in which case each $\rmH_\m^p(G,A)$ is of the form}
\[\rm{discrete}\oplus\rm{Euclidean};\]
\item \emph{$A$ locally compact and locally contractible and with trivial $G$-action, in which case $\rmH_\m^p(G,A)$ is of the form} \[\rm{discrete}\oplus\rm{Euclidean}.\]
\end{itemize}

\vspace{7pt}

(Recall that an Abelian topological group is locally compact and locally contractible if and only if it is a Lie group, although we shall generally prefer the former description in the sequel.)

The second set of questions concerns the relation between $\rmH_\m^\ast$ and the other theories $\rmH^\ast_\rm{ss}$ and $\rmH^\ast_\rm{cs}$.

Hofmann and Mostert~\cite{HofMos73} show that $\rmH_\rm{cs}^\ast$ for discrete target modules enjoys the same continuity properties as asserted by Theorems B and C, and more recently Flach has shown that $\rmH^\ast_\rm{ss}$ also has these continuity properties (Proposition 8.1 in~\cite{Fla08}).  Combined with Wigner's, Lichtenbaum's and Flach's results that these theories coincide for Lie groups and discrete target modules (see Theorem 4 in~\cite{Wig73}, Section 2 of~\cite{Lic09} and Proposition 5.2 in~\cite{Fla08}), this proves that all three theories coincide for compact base groups and discrete targets.

Using this special case, it is then possible to analyse more general locally compact second countable groups by applying the Lyndon-Hochschild-Serre spectral sequence to the presentation of such a group implied by the Gleason-Montgomery-Zippin Theorem on the resolution of Hilbert's Fifth Problem.  Since this spectral sequence makes the technical requirement that the cohomology groups of the kernel of a group extension be Hausdorff, Theorem D plays an important r\^ole in this extension of the comparison results between the various theories.

Our main new comparison result is the following.

\vspace{7pt}

\noindent\textbf{Theorem E}\quad \emph{If $G$ is a locally compact group and $A$ is a $G$-module which is either Fr\'echet or locally compact and locally contractible, then \[\rmH_\m^\ast(G,A)\cong \rmH_\rm{ss}^\ast(G,A).\]}

\emph{In addition, if $A$ is Fr\'echet then this agrees with $\rmH^\ast_\rm{cts}(G,A)$; if $A$ is locally compact and locally contractible, then it agrees with $\rmH^\ast_\rm{Seg}(G,A)$, since such an $A$ lies in Segal's category of modules; and if $A$ is discrete, then these also agree with $\rmH^\ast_\rm{cs}(G,A)$.}

\vspace{7pt}

We also offer a simple example in which both $G$ and $A$ are compact to show that $\rmH_\m^\ast$ does not always coincide with either $\rmH^\ast_\rm{ss}$ or $\rmH^\ast_\rm{cs}$.

In comparing $\rmH_\m^\ast$, $\rmH^\ast_\rm{cts}$, $\rmH^\ast_\rm{ss}$ and $\rmH^\ast_\rm{Seg}$, one should note that each has the virtue of being the unique solution to a universal problem --- namely an effaceable cohomological functor on its category of definition. Indeed, for comparing just $\rmH^\ast_\m$ and $\rmH^\ast_\rm{cts}$ one may choose the categories to be the same, and the difference appears only in what are the permitted `short exact sequences' to which the functor must assign long exact sequences.  (Hence this difference is very much in the spirit of `relative homological algebra': see, for instance, Section VI.2 of Brown~\cite{Bro82} for a more classical example.)

In the case of $\rmH_\m^\ast(G,\cdot)$ this category is $\sfP(G)$ and all short exact sequences as in~(\ref{eq:sh-ex}) are allowed.  This provides the natural setting for applications to functional analysis or representation theory, and so from the viewpoint of such applications the measurable-cochains theory is distinguished. It also has the virtue of having an equivalent definition without mention of cocycles in terms of equivalence classes of Yoneda-type multiple extensions of elements of $\sfP(G)$. However, one disadvantage of $\rmH_\m^\ast$ is that computations of it are quite difficult in degrees greater than two. On the other hand $\rmH^\ast_\rm{ss}$, while it enjoys its universal property only in a more abstract category of semi-simplicial sheaves, is defined using a spectral sequence in sheaf cohomology that is already a powerful computational tool. Similarly $\rmH^\ast_\rm{cs}$, which has no obvious universality properties, can sometimes be computed very readily using tools from algebraic topology (as in~\cite{HofMos73}, for example). Hence Theorem E, which expands the known areas of agreement between $\rmH_\m^\ast$, $\rmH^\ast_\rm{ss}$ and $\rmH^\ast_\rm{cs}$, strengthens the usefulness of all three theories.

In fact, in the present paper we will already make use of this connexion in the proofs of some parts of Theorem D.  This does not make the relationship between Theorem D and Theorem E circular: rather, we will first prove some special cases of Theorem E without using the LHS spectral sequence, then use these to obtain Theorem D, and then with this in hand return to more general cases of Theorem E.

\subsection{Another application}

Although this paper is purely about cohomology theories for topological groups, we note in passing that it was originally motivated by a very concrete application in ergodic theory, to questions concerning the structure of `characteristic factors' for certain nonconventional ergodic averages.  In the case recently treated in~\cite{Aus--lindeppleasant1,Aus--lindeppleasant2}, part of this structure could be characterized using a family of Borel $3$-cocycles on a compact Abelian group $W$ taking values in the $W$-module $\cal{E}(W)$ of all affine $\bbT$-valued maps on $W$, endowed with the rotation action of $W$.  These $3$-cocycles emerge in the description of some isometric circle extension of a measurable probability-preserving $\bbZ^2$-action by rotations on $W$, and more specifically as the obstructions to the factorizability of some $2$-cocycle that actually corresponds to an extension of acting groups.  It is shown in~\cite{Aus--lindeppleasant2} that all examples of such systems can be obtained as inverse limits of finite-dimensional examples (in a certain natural sense), and these in turn can always be assembled by performing some standard manipulations on a special class of examples (joinings of partially invariant systems and two-step nilsystems).  The continuity results of the present paper were at the heart of the reduction to finite-dimensional examples in that work, and so provided a crucial ingredient needed in the proof of the main structure theorem in~\cite{Aus--lindeppleasant2}.  We will not give a more complete introduction to these questions here, but the whole story can be found in those papers.

\section{Preliminaries}\label{sec:prelim}

Throughout this paper we will work with the presentation of the cohomology theory $\rmH_\m^\ast$ using the inhomogeneous bar resolution, and refer to~\cite{Moo64(gr-cohomI-II),Moo76(gr-cohomIII),Moo76(gr-cohomIV)} for the relation of this to other definitions.

Suppose that $G$ is a locally compact second countable group, and let $m_G$ denote a left-invariant Haar measure on $G$, normalized to be a probability if $G$ is compact.  These assumptions will now stand for the rest of this paper unless explicitly contradicted. A \textbf{Polish Abelian $G$-module} is a triple $(A,\rho,T)$ in which $A$ is a Polish Abelian group with translation-invariant metric $\rho$ and $T:G\actson A$ is an action by continuous automorphisms.  It is \textbf{Fr\'echet} if $A$ is also a separable and locally convex real topological vector space in its Polish topology, and it is \textbf{Euclidean} if in addition it is finite dimensional.  It will later prove helpful to keep the metric and action explicit.  Since $T$ is continuous, if $G$ is compact then by averaging the shifted metrics $\rho(T^g\cdot,T^g\cdot)$ if necessary we may assume that $\rho$ is also $T$-invariant.

If $A$ is a Polish $G$-module and $p\geq 0$ then we write
$\cal{C}(G^p,A)$ for the $G$-module of
Haar-a.e. equivalence classes of Borel maps $G^p\to A$ with the
topology of convergence in probability (defined using the topology
of $A$).  Since we assume $G$ is second countable this topology on $\cal{C}(G^p,A)$ is also Polish.  Beware that the notation `$\C(G^p,A)$' indicates `cochains', as usual in group cohomology, and not `continuous functions' --- continuity will be marked by a subscript, as will be exhibited shortly.

When $p = 0$ we interpret $\C(G^p,A)$ as $A$, and when $p < 0$ we interpret it as the trivial group $(0)$.  When necessary we will equip it with the \textbf{diagonal action}
\[(R^gf)(g_1,g_2,\ldots,g_p) := T^g(f(g^{-1}g_1,g^{-1}g_2,\ldots,g^{-1}g_p)).\]
We denote the identity in $A$ or in any $\C(G^p,A)$ by $0$.

We also write $\C^p(G,A) := \C(G^p,A)$, and define the coboundary maps $d:\C^p(G,A)\to \C^{p+1}(G,A)$ by
\begin{eqnarray}\label{eq:bdrydfn}
d\phi(g_1,g_2,\ldots,g_{p+1}) &:=& T^{g_1}(\phi(g_2,g_3,\ldots,g_{p+1}))\nonumber\\ &&+ \sum_{i=1}^{p}(-1)^i\phi(g_1,g_2,\ldots,g_ig_{i+1},\ldots,g_{p+1})\nonumber\\ &&+ (-1)^{p+1}\phi(g_1,g_2,\ldots,g_p)
\end{eqnarray}
(where this and all similar later equations are to be understood as holding Haar-almost-everywhere; issues surrounding this point are discussed carefully in Section 4 of~\cite{Moo76(gr-cohomIII)}, and we will not dwell on them here). We define $\Z^p(G,A) := \ker d|_{\C^p(G,A)}$ and $\B^p(G,A) := \img d|_{\C^{p-1}(G,A)}$. As usual, one verifies that $d^2 = 0$, and so $\B^p(G,A)\subseteq \Z^p(G,A)$ and we can define
\[\rmH_\m^p(G,A) := \frac{\Z^p(G,A)}{\B^p(G,A)}:\]
these are the \textbf{measurable cohomology groups for $G$ with coefficients in $A$}.

Exactly analogously, we define $\C^p_\rm{cts}(G,A)$ to consist of cochains $G^p\to A$ that are continuous, and within it the subspaces $\Z^p_\rm{cts}(G,A)$ and $\B^p_\rm{cts}(G,A)$ of continuous cocycles and coboundaries.  The resulting cohomology groups
\[\rmH^p_\rm{cts}(G,A) := \frac{\Z_\rm{cts}^p(G,A)}{\B^p_\rm{cts}(G,A)}\]
are the \textbf{continuous cohomology groups for $G$ with coefficients in $A$}.

The various other cohomology theories for topological groups mentioned in the Introduction will be quickly recalled when we begin their analysis in Section~\ref{sec:other-theories}.

If $\pi:G'\to G$ is a homomorphism of groups and $\psi \in \C^p(G,A)$ then we write
\[\pi^{\times p}:= \pi\times \pi\times\cdots\times \pi:(G')^p\to G^p,\]
so that $\psi \circ\pi^{\times p} \in \C^p(G',A)$.  Clearly \[\Z^p(G,A)\circ \pi^{\times p}\subseteq \Z^p(G',A)\quad\hbox{and}\quad \B^p(G,A)\circ \pi^{\times p}\subseteq \B^p(G',A),\] so this lifting homomorphism has a quotient $\inf_\pi:\rmH_\m^p(G,A)\to \rmH_\m^p(G',A)$, referred to as the \textbf{inflation homomorphism} associated to $\pi$.

In case $G$ is compact, a Borel map $\psi:G^p\to A$ is \textbf{$\eps$-small} for some $\eps > 0$ if
\[m_{G^p}\{\bf{g}\in G^p:\ \rho(0,\psi(\bf{g})) > \eps\} < \eps\]
(so this definition implicitly involves a choice of metric on $A$).  For two such maps $\phi$, $\psi$ we define
\[\rho_0(\phi,\psi) := \inf\{\eps > 0:\ \phi - \psi\ \hbox{is $\eps$-small}\}.\]
This is routinely verified to define a metric on $\C(G,A)$ which metrizes the topology of convergence in probability; when $A = \bbR$ this gives the usual F-space structure on $\C(G,\bbR)$, which space is often denoted by $L^0(G)$ in functional analysis. We also define the \textbf{uniform metric} $\rho_\infty$ on $\C(G^p,A)$ associated to $\rho$ by
\[\rho_\infty(\phi,\psi) := \rm{ess\,sup}_{\bf{g}\in G^p}\ \rho(\phi(\bf{g}),\psi(\bf{g}));\]
of course for arbitrary Borel maps this may take the value $+\infty$.

We will sometimes use some standard analyst's notation: the relation $B \lesssim_X A$ asserts that $B$ is bounded by the product of $A$ with some positive constant depending only on $X$, and similarly a quantity is $\rm{O}_X(A)$ if it is bounded by the product of $A$ with some positive constant depending only on $X$.

\section{Improving the regularity of cocycles}

\subsection{Recap of dimension-shifting}

Our later arguments will rely crucially on the procedure of dimension-shifting, which allows us to re-write one cohomology group as another of different degree (and with a different target module), and so by induction on degree gain access to algebraic properties that are manifest only in low degrees (usually degree one).

This possibility follows from the standard long exact sequence together with effacement.  Concretely, we will use the vanishing result that $\rmH_\m^p(G,\C(G,A)) = (0)$ for all $p\geq 1$ and for any $(A,\rho,T)$ when $\C(G,A)$ is equipped with the diagonal action (see Theorem 4 of~\cite{Moo76(gr-cohomIII)}).  Indeed, any $G$-module $A$ embeds into $\C(G,A)$ as the closed submodule of constant maps; let $\iota:A\into \C(G,A)$ be this embedding.  As a result of the vanishing, constructing the long exact sequence from the presentation
\[A\into \C(G,A)\onto \C(G,A)/\iota(A)\]
collapses to give a sequence of switchback maps that are isomorphisms $\rmH_\m^p(G,A) \cong\rmH_\m^{p-1}(G,\C(G,A)/\iota(A))$ for
all $p\geq 1$.

As is standard, the switchback map obtained above is implemented by a simple operator from $\Z^p(G,A)$ into $\C^{p-1}(G,\C(G,A))$: if $\psi \in \Z^p(G,A)$ then we define $Q\psi \in \C^{p-1}(G,\C(G,A))$ by
\[Q\psi(g_1,g_2,\ldots,g_{p-1})(h) := (-1)^p\psi(g_1,g_2,\ldots,g_{p-1},g_{p-1}^{-1}g_{p-2}^{-1}\cdots g_1^{-1}h).\]
Now a straightforward manipulation of the equation $d\psi = 0$ shows that $d(Q\psi)(g_1,\ldots,g_p)$ is actually the constant-valued map $h\mapsto \psi(g_1,\ldots,g_p)\in A$.  We record this here for convenience: for $(g_1,g_2,\ldots,g_p)\in G^p$ and $h\in G$ we have from the definition~(\ref{eq:bdrydfn})
\begin{eqnarray*}
&&d(Q\psi)(g_1,g_2,\ldots,g_p)(h)\\ &&= R^{g_1}\big(Q\psi(g_2,\ldots,g_p)\big)(h)\\
&&\quad + \sum_{i=1}^{p-1}(-1)^iQ\psi(g_1,g_2,\ldots,g_ig_{i+1},\ldots,g_p)(h) + (-1)^pQ\psi(g_1,g_2,\ldots,g_{p-1})(h)\\
&&= (-1)^pT^{g_1}(\psi(g_2,g_3,\ldots,g_p,g_p^{-1}g_{p-1}^{-1}\cdots g^{-1}_2(g_1^{-1}h))\\
&&\quad + (-1)^p\sum_{i=1}^{p-1}(-1)^i\psi(g_1,g_2,\ldots,g_ig_{i+1},\ldots,g_p,g_p^{-1}g_{p-1}^{-1}\cdots g_1^{-1}h)\\
&&\quad + (-1)^{2p}\psi(g_1,g_2,\ldots,g_{p-1},g_p(g_p^{-1}g_{p-1}^{-1}\ldots,g_1^{-1}h))\\
&&= (-1)^pd\psi(g_1,g_2,\ldots,g_p,g_p^{-1}g_{p-1}^{-1}\cdots g_1^{-1}h) + \psi(g_1,g_2,\ldots,g_p)\\
&&= \psi(g_1,g_2,\ldots,g_p).
\end{eqnarray*}

Thus, the image of $Q\psi$ under the quotient $\C(G,A)\onto \C(G,A)/\iota(A)$ defines a class in $\rmH_\m^{p-1}(G,\C(G,A)/\iota(A))$, and the usual diagram chase shows that it depends only on $\psi + \B^p(G,A) \in \rmH_\m^p(G,A)$: this new class is the image of $\psi + \B^p(G,A)$ under the switchback isomorphism.

\subsection{The regularizing argument: qualitative version}

The measurable crossed homomorphisms that appear in degree-$1$ cohomology are automatically continuous, and the first benefit of dimension-shifting for our work is that it allows us to derive consequences of this in higher degrees.  We will now prove a basic result giving a sense in which all cohomology classes can be represented by cocycles having some additional regularity, not possessed by arbitrary measurable cochains.

\begin{dfn}
A Borel map $f:X\to Y$ from a locally compact space to a Polish space is \textbf{locally totally bounded} if for any compact $K\subseteq X$ the image $f(K)$ is precompact in $Y$.
\end{dfn}

We will need the following classical result from the topology of metric spaces: it follows, for instance, from Proposition 18 in Section IX.2 of Bourbaki~\cite{Bourb89}.

\begin{lem}
If $A \leq B$ is an inclusion of Polish Abelian groups (so $A$ is closed in $B$) and $q:B\to B/A$ is the quotient map, then for any compact subset $K \subseteq B/A$ there is a compact subset $L \subseteq B$ with $q(L) \supseteq K$. \qed
\end{lem}

\begin{prop}\label{prop:basic-smoothness}
For any locally compact, second countable $G$ and Polish $G$-module $(A,\rho,T)$, any cohomology class in $\rmH_\m^p(G,A)$ has a representative $\s \in \Z^p(G,A)$ which is locally totally bounded.
\end{prop}

\begin{proof}
This follows by a dimension-shifting induction.  When $p=1$ it is immediate from the automatic continuity of $1$-cocycles, so suppose now that $p\geq 2$ and $\psi:G^p\to A$ is a $p$-cocycle.

By dimension-shifting there is a cochain $\l\in \C^{p-1}(G,\C(G,A))$ with $\psi = d\l$, so that the image $\bar{\l}$ of $\l$ under the quotient onto $\C(G,A)/\iota(A)$ is a $(p-1)$-cocycle.  Therefore by the inductive hypothesis it is equal to $\bar{\k} + d\bar{\a}$ for some $\bar{\a}:G^{p-2}\to \C(G,A)/\iota(A)$ and some locally totally bounded cocycle $\bar{\k}$.

By choosing a Borel partition of $G^{p-1}$ that is countable, locally finite and has each cell precompact, and then applying the preceding lemma on each cell, there are measurable lifts $\k$ and $\a$ of $\bar{\k}$ and $\bar{\a}$ so that $\k$ is still locally totally bounded.  It follows that
\[\l = \k + d\a + \b\]
for some $\b:G^{p-1}\to A$, and hence
\[\psi = d\k + d\b.\]
Now, on the one hand, $d\k = \psi - d\b$ must takes values in the subgroup $\iota(A) \subset \C(G,A)$, and so is identified with an element of $\Z^p(G,A)$ cohomologous to $\psi$.  On the other, for any compact $K \subseteq G$ one has the property that
\[d\k(K^p) \subseteq T^K(\k(K^{p-1})) - \k(K^{p-1}) + \ldots + (-1)^{p+1}\k(K^{p-1}),\]
so this is precompact as a subset of $\C(G,A)$ by the local total boundedness of $\k$.  Since the topology of the closed subgroup $\iota(A) \subseteq \C(G,A)$ agrees with the subspace topology, this implies that $\s:= d\k$ is locally totally bounded, so the induction continues. \qed
\end{proof}

\subsection{The regularizing argument: quantitative version}

Proposition~\ref{prop:basic-smoothness} is already enough to derive some consequences on the structure of the groups $\rmH_\m^\ast$ or the comparison with other cohomology theories, but the continuity results of Theorems B and C for discrete or toral targets are more delicate.  These will rest on a quantitative analog of Proposition~\ref{prop:basic-smoothness} which in some settings will give us an additional ability to prove the triviality of cocycles that are quantitatively `small enough'.

We will need the following elementary estimate relating the sizes of $\psi$ and $Q\psi$, where $Q$ is the explicit dimension-shifting operator introduced at the beginning of this section.

\begin{lem}\label{lem:dim-shift-reg-1}
If $G$ is compact and $\psi \in \Z^p(G,A)$ is $\eps$-small with $0 < \eps < 1$, then $Q\psi$ is $\sqrt{\eps}$-small as a map from $G^{p-1}$ to the module $(\C(G,A),\rho_0,R)$: that is,
\[m_{G^{p-1}}\big\{\bf{g} \in G^{p-1}:\ m_G\{x \in G:\ \rho(0,Q\psi(\bf{g})(x)) \geq \sqrt{\eps}\} \geq \sqrt{\eps}\big\} < \sqrt{\eps}.\]
\end{lem}

\begin{proof}
The above measure is simply
\begin{multline*}
m_{G^{p-1}}\big\{(g_1,g_2,\ldots,g_{p-1}) \in G^{p-1}:\\ m_G\{x \in G:\ \rho(0,\psi(g_1,g_2,\ldots,g_{p-1},g_{p-1}^{-1}\cdots g_1^{-1}x) \geq \sqrt{\eps}\} \geq \sqrt{\eps}\big\},
\end{multline*}
so since the map
\[G^p \to G^p:(g_1,g_2,\ldots,g_{p-1},x)\mapsto (g_1,g_2,\ldots,g_{p-1},g_{p-1}^{-1}\cdots g_1^{-1}x)\]
preserves Haar measure this follows directly from Fubini's Theorem. \qed
\end{proof}

The next lemma is a quantitative analog of the standard result that crossed homomorphisms are always continuous.

\begin{lem}\label{lem:L0-to-Linf}
Suppose that $G$ is compact, that $(A,\rho,T)$ is a Polish $G$-module and that $\a:G\to A$ is a Borel crossed homomorphism.  If $\a$ is $\eps$-small for some $\eps < \frac{1}{2}$ then $\rho_\infty(0,\a) \leq 2\eps$.
\end{lem}

\begin{proof}
Consider the crossed homomorphism equation
\[T^h\a(g) = \a(hg) - \a(h).\]
Let $E:= \{g:\ \rho(0,\a(g)) \leq \eps\}$, so by assumption $m_G(E) > \frac{1}{2}$.  Therefore $m_G(E\cap Eg^{-1}) > 0$ for any $g \in G$, and so for any $g$ we can find $h \in E$ so that also $hg \in E$.  Now since $\a(h)$ and $\a(hg)$ are both within distance $\eps$ of $0$ and $T$ preserves $\rho$, the above equation gives that $\a(g)$ is within distance $2\eps$ of $0$. \qed
\end{proof}

\begin{prop}\label{prop:L0-to-Linf}
There is a sequence of absolute constants $\eta_p > 0$ for which the following holds. Suppose that $G$ is compact, that $(A,\rho,T)$ is a Polish $G$-module, that $p\geq 1$ and that $\psi \in \Z^p(G,A)$ is $\eps$-small for some $\eps \leq \eta_p$. Then $\psi = \phi + d\l$ for some $\l \in \C^{p-1}(G,A)$ that is $\rm{O}_p(\eps^{2^{-p}})$-small and some $\phi \in \Z^p(G,A)$ with $\rho_\infty(0,\phi) \lesssim_p \eps^{2^{-p}}$.
\end{prop}

\begin{proof}
The construction of $(\eta_p)_{p\geq 1}$ and the proof proceed by induction on $p$.

\emph{Base clause: $p=1$}\quad If we choose $\eta_1 < \frac{1}{2}$ then in this case any $\psi\in\Z^1(G,A)$ is a crossed homomorphism and $\eps < \frac{1}{2}$, so it must already satisfy $\rho_\infty(0,\psi) < 2\eps$ by Lemma~\ref{lem:L0-to-Linf}.

\emph{Recursion clause}\quad Suppose we know the assertion in all degrees up to some $p\geq 1$ and for some $\eta_1$, $\eta_2$, \ldots, $\eta_p$, and we wish to prove it in degree $p+1$.  Given an $\eps$-small $\psi\in\Z^{p+1}(G,A)$, the dimension-shifting operator gives $Q\psi \in \C^p(G,\C(G,A))$ that is $\sqrt{\eps}$-small and is such that its image $\ol{Q\psi}$ upon quotienting by $\iota(A)$ is a member of $\Z^p(G,\C(G,A)/\iota(A))$.  Of course $\ol{Q\psi}$ is also $\sqrt{\eps}$-small, so applying the inductive hypothesis to $\ol{Q\psi}$ we see that it is of the form $\bar{\k} + d\bar{\a}$ for some $\rm{O}_p((\sqrt{\eps})^{2^{-p+1}})$-small $\bar{\a} \in \C^{p-1}(G,\C(G,A)/\iota(A))$ and some $\bar{\k} \in \Z^p(G,\C(G,A)/\iota(A))$ for which
\[\rho_\infty(0,\bar{\k}) := \rm{ess\,sup}_{\bf{g} \in G^p}\inf\{\rho_0(0,\k):\ \k\in \bar{\k}(\bf{g})\} \lesssim_p(\sqrt{\eps})^{2^{-p+1}} = \eps^{2^{-p}}.\]

Making two applications of the measurable selection theorem, we may select some $\rm{O}_p(\eps^{2^{-p}})$-small $\a$ lifting $\bar{\a}$ and some $\k \in \C^{p-1}(G,\C(G,A))$ lifting $\bar{\k}$ such that
\[\rho_\infty(0,\k) := \rm{ess\,sup}_{\bf{g} \in G^p}\rho_0(0,\k(\bf{g})) \lesssim_p\eps^{2^{-p}},\]
and for these it follows that $Q\psi = \k + d\a + \l$ for some $\l\in \C^p(G,\C(G,A))$ that takes values in the subgroup $\iota(A)$ of constant maps.

Since $Q\psi$, $\k$ and $d\a$ are all $\rm{O}_p(\eps^{2^{-p}})$-small as maps from $G^p$ into the module $(\C(G,A),\rho_0,R)$, it follows that $\l$ is also $\delta$-small for some $\delta \lesssim_p \eps^{2^{-p}}$. Provided we chose $\eps < \eta_{p+1}$ for some $\eta_{p+1}$ sufficiently small, this implies that $\delta < 1$, and hence that there is some subset of $\bf{g} \in G^p$ of measure at least $1-\delta$ for which the map $\l(\bf{g}) \in \C(G,A)$ is $\delta$-small, and so takes values within $\delta$ of $0 \in A$ on a subset of $G$ of measure at least $1 - \delta$. However, since $\l:G^p\to\C(G,A)$ actually takes values in the subgroup $\iota(A)$ of constant-valued maps, we may therefore identify it with a $\delta$-small member of $\C^p(G,A)$.  Note that the assumption that $\delta < 1$ was crucial here in justifying the implication that a constant-valued function $G\to A$ is $\delta$-small as a member of $\C(G,A)$ only if its constant value lies within $\delta$ of $0 \in A$.

Finally applying the coboundary operator gives
\[\psi = d(Q\psi) = d\k + d\l\]
where $\l$ is $\rm{O}_p(\eps^{2^{-p}})$-small, and if we set $\phi := d\k$ then
\begin{itemize}
\item on the one hand we have
\[\rho_\infty(0,\phi) = \rm{ess\,\sup}_{\bf{g} \in G^{p+1}}\rho_0(0,\phi(\bf{g})) \lesssim_p \eps^{2^{-p}},\]
since for every $\bf{g}$ the map $\phi(\bf{g}) \in \C(G,A)$ is a sum of $p+1$ different values of $\k$, all of which satisfy such an essential supremum bound;
\item and on the other both $\psi$ and $d\l$ take values in the subgroup $\iota(A) \subset \C(G,A)$, and hence the same must be true of $\phi$.
\end{itemize}
Therefore, once again making the proviso that $\eta_{p+1}$ be sufficiently small to guarantee that $\rho_\infty(0,\phi) < 1$, we deduce that the constant values taken by $\phi$ must lie within $\rm{O}_p(\eps^{2^{-p}})$ of $0$ almost surely, and so the induction continues to $p+1$. \qed
\end{proof}

\noindent\emph{Remark}\quad A fairly crude check shows that the sequence $\eta_p := \frac{1}{100(p!)^2}$ is certainly small enough. \fin

\section{Fr\'echet modules and continuous cocycles}

Theorem A follows quite quickly from Proposition~\ref{prop:basic-smoothness}.

\vspace{7pt}

\noindent\emph{Proof of Theorem A}\quad This is known in degree $1$ by the automatic continuity of crossed homomorphisms.  We will show that for any $p > 1$ and any Fr\'echet $G$-module $A$, any cohomology class in $\rmH_\m^p(G,A)$ has a representative cocycle $\psi \in \Z^p(G,A)$ that is effaced by the inclusion
\begin{eqnarray}\label{eq:inc-to-cts}
\iota:A\into \C_\rm{cts}(G,A),
\end{eqnarray}
where the latter is given its compact-open topology. This shows that $\rmH_\m^\ast$ remains effaceable if it is restricted to the category of Fr\'echet $G$-modules, since $\C_\rm{cts}(G,A)$ is still an object of this category (unlike $\C(G,A)$, which is in general a non-locally-convex F-space).  The continuous-cochains theory $\rmH^\ast_\rm{cts}$ is also effaceable under the same inclusion, and it also has long exact sequences because quotients of Fr\'echet spaces admit continuous lifts by the Bartle-Graves-Michael selection theorems (see, for instance, Section 1.3 of Benyamini and Lindenstrauss~\cite{BenyLin00}). Therefore, since the two theories are known to agree in degree $1$, the Buchsbaum criterion carries this agreement up to all degrees, where it is clear that the isomorphism is simply given by the obvious inclusion $\Z^p_\rm{cts}(G,A)\into \Z^p(G,A)$.

To prove effacement by~(\ref{eq:inc-to-cts}), we first deduce from Proposition~\ref{prop:basic-smoothness} that we may assume $\psi$ is locally totally bounded.  Let $\eta:G\to [0,\infty)$ be a continuous bump function: a compactly-supported continuous function with $\int_G \eta\,\d m_G = 1$, where $m_G$ is a left-invariant Haar measure on $G$.  We can now perform the classic averaging trick from bounded cohomology by integrating against $\eta$: define $\k:G^{p-1}\to \C(G,A)$ by
\[\k(g_1,g_2,\ldots,g_{p-1})(g) := (-1)^p \int_G \psi(g_1,\ldots,g_{p-1},g_{p-1}^{-1}\cdots g_1^{-1}gh) \eta(h)\, m_G(\d h).\]
This integral exists in the strong sense of Bochner by results of Grothendieck~\cite{Groth55} (see also Thomas~\cite{Tho75}), since the integrand is compactly supported and takes values in a compact set, so it is certainly absolutely summable in Grothendieck's sense. The continuity of $\eta$ and left-invariance of $m_G$ imply that $\k$ takes values in the subgroup $\C_\rm{cts}(G,A)$, and now the usual calculation gives
\begin{eqnarray*}
&&d\k(g_1,g_2,\ldots,g_p)(g)\\
&&= (-1)^p T^{g_1}\Big(\int_G \psi(g_2,\ldots,g_p,g_p^{-1}...g_2^{-1}(g_1^{-1}gh)) \psi(h)\,m_G(\d h) \Big)\\
&&\quad\quad - (-1)^p \int_G \psi(g_1g_2,\ldots,g_p,g_p^{-1}\cdots g_2^{-1}(g_1^{-1}gh)) \eta(h)\,m_G(\d h)\\
&&\quad\quad + (-1)^p \int_G \psi(g_1,g_2g_3,\ldots,g_p,g_p^{-1}\cdots g_2^{-1}(g_1^{-1}gh)) \eta(h)\,m_G(\d h)\\
&&\quad\quad - \cdots + (-1)^{2p} \int_G \psi(g_1,\ldots,g_{p-1},g_p(g_p^{-1}\cdots g_1^{-1}gh)) \eta(h)\,m_G(\d h)
\end{eqnarray*}
\begin{eqnarray*}
&&=\int_G \big((-1)^p d\psi(g_1,\ldots,g_p,g_p^{-1}\cdots g_1^{-1}gh) + \psi(g_1,\ldots,g_p) \big) \eta(h)\,m_G(\d h)\\
&&= \int_G \psi(g_1,\ldots,g_p) \eta(h)\,m_G(\d h)\\
&&= \psi(g_1,\ldots,g_p),
\end{eqnarray*}
as required.

In case $G$ is compact, we may let $\eta$ be the constant function $1$: in this case the cochain $\kappa$ obtained above simply takes values in $A$, and so $\psi = d\k$ is always a coboundary and hence $\rmH_\m^p(G,A) = (0)$. \qed

\vspace{7pt}

\noindent\emph{Remark}\quad Smoothing out cocycles by integrating against a bump function already has a precedent in Hochschild and Mostow~\cite{HocMos62}, where it is used for the slightly different task of comparing continuous with differentiable cocycles: see their Lemma 5.2 and Theorem 5.1. \fin

\vspace{7pt}

Of course, among non-Fr\'echet modules there are easy examples for which the analog of Theorem A fails: for example, $\rmH_\m^2(\bbT,\bbZ)$ is isomorphic to $\bbZ$ (as may be proved directly using the presentation $\bbZ\into \bbR\onto \bbT$ and the long exact sequence, or using Theorem E below), but any continuous map $\bbT^2\to\bbZ$ must be constant and so $\rmH^2_{\rm{cts}}(\bbT,\bbZ) = (0)$.

Some larger classes than the Fr\'echet modules still present an interesting question, however.  The above proof makes essential use of the property of Fr\'echet spaces that their quotients admit continuous sections.  This result may not be available for more general non-locally-convex F-spaces.  Indeed, rather strikingly, the following seems to be unknown:

\begin{ques}\label{ques:R-in-L0}
Consider $\bbR$ identified with the subspace of constant functions in $L^0([0,1])$, the space of a.e.-equivalence classes of measurable functions $[0,1]\to \bbR$.  Does it have a continuous cross-section for the topology of convergence in probability?
\end{ques}

This analytic issue gives rise to a corresponding question in cohomology.

\begin{ques}\label{ques:bad-compact-and-contractible}
Is there a non-locally-convex F-space $A$, or other contractible Polish group, which admits a continuous action of a compact metric group $G$ such that $\rmH_\m^p(G,A) \neq (0)$ for some $p \geq 1$?
\end{ques}

\section{Behaviour under inverse and direct limits}

\subsection{Proofs of Theorems B and C}

We next approach Theorems B and C.  These will use the quantitative version of our basic smoothing argument via the following two notions.

\begin{dfn}
A Polish $G$-module $(A,\rho,T)$ is \textbf{regular under smallness assumptions} if there is a sequence of positive constants $\delta_p > 0$ such that for any $p\geq 1$ and $\eps > 0$ there is some $\delta \in (0,\delta_p)$ such that any cocycle $\psi \in \Z^p(G,A)$ that is $\delta$-small is of the form $\psi = d\phi$ for some $\phi \in \C^{p-1}(G,A)$ that is $\eps$-small.

Similarly, it is \textbf{regular under uniform approximation} if there is a sequence of positive constants $\delta'_p > 0$ such that for any $p\geq 1$ and $\eps > 0$ there is some $\delta \in (0,\delta'_p)$ such that any cocycle $\psi \in \Z^p(G,A)$ having $\rho_\infty(0,\psi) \leq \delta$ is of the form $\psi = d\phi$ for some $\phi \in \C^{p-1}(G,A)$ with $\rho_\infty(0,\phi) \leq \eps$.
\end{dfn}

The point to all our dimension-shifting is that we can promote the latter of these properties to the former.

\begin{lem}\label{lem:promoting-regularity-assumptions}
If $G$ is compact and $(A,\rho,T)$ is regular under uniform approximation with error tolerances $\delta_p'$, then it is regular under smallness assumptions with error tolerances $\delta_p$ depending only on the $\delta_p'$.
\end{lem}

\begin{proof}
If $\psi \in \Z^p(G,A)$ is $\delta$-small for some $\delta \in (0,\eta_p)$, then Proposition~\ref{prop:L0-to-Linf} gives that $\psi = \phi + d\l$ for some $\rm{O}_p(\delta^{2^{-p}})$-small $\l$ and $\phi$ with $\rho_\infty(0,\phi) \lesssim_p \delta^{2^{-p}}$.  Hence provided we choose $\delta_p$ small enough to guarantee that this implies $\rho_\infty(0,\phi) \leq \delta'_p$, we may now apply the regularity of $A$ under uniform approximation to $\phi$ to deduce that $\phi$ and hence also $\psi$ are the boundaries of small cochains, as required. \qed
\end{proof}

It is also worth recording at once the following easy connexion.

\begin{lem}\label{lem:reg-implies-Haus}
If $G$ is compact and $(A,\rho,T)$ is regular under smallness assumptions then $\B^p(G,A)$ is a clopen subgroup of $\C^p(G,A)$ for every $p\geq 0$ for the topologies of convergence in probability, and hence the groups $\rmH_\m^p(G,A)$ are discrete and countable in their quotient topologies.
\end{lem}

\begin{proof}
The result is trivial for $p=0$ since $\B^0(G,A) := (0)$, and regularity under smallness assumptions implies that $\B^p(G,A)$ is an open subset of $\C^p(G,A)$ for each $p$.  Since open subgroups are necessarily also closed, this implies discreteness of the quotient.  Countability now follows because $\C^p(G,A)$ is Polish (recall that we always assume $G$ is second countable), and hence its image $\C^p(G,A)/\B^p(G,A)$ must still be separable for the quotient topology. \qed
\end{proof}

We are now just one auxiliary proposition away from a proof of Theorem B.

\begin{prop}\label{prop:inv-lims}
Suppose that $(G_m)_m$, $(\pi^m_k)_{m\geq k}$ is an inverse system of compact metric groups with inverse limit $G$, $(\pi_m)_m$, and that $(A,\rho,T)$ is a Polish $G$-module with action factorizing through each $\pi_m$ and with the property that it is regular under smallness assumptions with the same error tolerances $\delta_p$ when interpreted as a $G_m$-module for every $m$. Then the direct limit of inflation homomorphisms
\[\lim_{m\rightarrow}\rm{inf}^p_{\pi_m}:\lim_{m\rightarrow}\rmH_\m^p(G_m,A)\to \rmH_\m^p(G,A)\]
is an isomorphism for all $p \geq 1$ (where we interpret $A$ as a $G$- or $G_m$-module by factorizing the action through $\pi_m$ if necessary).
\end{prop}

\begin{proof}\quad\emph{Surjectivity}\quad Suppose that $\psi \in \Z^p(G,A) \subseteq \C^p(G,A)$ and that $\eps,\delta > 0$.  Since the epimorphisms $\pi_m$ together generate the whole Borel $\s$-algebra of $G$ there are some $m$ and $\psi_1 \in \C^p(G_m,A)$ such that $\psi - \psi_1\circ \pi_m^{\times p}$ is $\delta$-small.

A priori we are not necessarily able to take $\psi_1$ to be a cocycle.  However, setting $\psi_2 := \psi - \psi_1\circ \pi_m^{\times p}$ and $\l := d\psi_2 = -d\psi_1\circ \pi_m^{p+1}$ (using that $d\psi = 0$), it follows that $\l$ is both $\rm{O}_p(\delta)$-small and is lifted from some $\l' \in \Z^{p+1}(G_m,A)$, which must therefore also be $\rm{O}_p(\delta)$-small.  Provided $\delta$ was chosen sufficiently small depending on $p$ and $\eps$, this new smallness bound will be less than $\delta_{p+1}$ and regularity under smallness assumptions will give that $\l' = d\psi_3$ for some $\eps$-small $\psi_3 \in \C^p(G_m,A)$.

Putting these approximations together we deduce that $\psi_4 := \psi_1+\psi_3$ is a cocycle and that $\psi - \psi_4\circ \pi_m^{\times p}$ is $(\delta + \eps)$-small, so, making one last appeal to having chosen $\eps$ and $\delta$ sufficiently small, the regularity of $A$ under smallness assumptions now gives that $\psi$ is actually cohomologous to the finite-dimensional cocycle $\psi_4\circ\pi_m^{\times p}$.  This proves surjectivity.

\emph{Injectivity}\quad Now suppose that $\psi \in \Z^p(G_m,A)$ is such that $\psi\circ \pi_m^{\times p} = d\k$ for some $\k \in \C^{p-1}(G,A)$; we must show that we can obtain a similar coboundary equation upon lifting only to some finite $m' \geq  m$.  For any $\delta > 0$ we may pick some finite $m'\geq m$ and $\k_1 \in \C^{p-1}(G_{m'},A)$ such that $\k_2 := \k - \k_1\circ\pi_{m'}^{\times (p-1)}$ is $\delta$-small.  It follows that
\[\phi := d\k_2 = \psi\circ \pi_m^{\times p} - d\k_1\circ \pi_{m'}^{\times p}\]
is an $\rm{O}_p(\delta)$-small member of $\B^p(G,A)$ that is lifted from $\C^p(G_{m'},A)$, and hence that
\[\phi' := \psi\circ (\pi^{m'}_m)^{\times p} - d\k_1\]
is an $\rm{O}_p(\delta)$-small member of $\Z^p(G_{m'},A)$.   Therefore provided we chose $\delta$ sufficiently small, the regularity of $A$ under smallness assumptions shows that it is actually a coboundary, and hence the same is true of $\psi\circ(\pi^{m'}_m)^{\times p}$, as required. \qed
\end{proof}

\noindent\emph{Proof of Theorem B}\quad\emph{Discrete modules}\quad Combining Lemma~\ref{lem:promoting-regularity-assumptions} and Proposition~\ref{prop:inv-lims}, it suffices to prove that any discrete Abelian group $A$ is regular under uniform approximation with error tolerances not depending on the compact acting group. However, if we pick each $\delta_p$ to equal some $\delta > 0$ so small that $\{a \in A:\ \rho(0,a) \leq \delta\} = \{0\}$ in $A$, then the only $\psi \in \C^p(G,A)$ with $\rho_\infty(0,\psi) \leq \delta_p$ are the maps with constant value $0$, for which the assertion is trivial.

\emph{Tori}\quad If $A = \bbT^d = (\bbR/\bbZ)^d$, then our easiest approach is to use the case of discrete groups via the switchback isomorphisms $\rmH_\m^p(G,\bbT^d) \cong \rmH_\m^{p+1}(G,\bbZ^d)$ (which are easily seen to respect our regularity assumptions) resulting from the long exact sequence and the vanishing $\rmH_\m^p(G,\bbR^d) = (0)$ that we obtained in Theorem A. \qed

\vspace{7pt}

A similar argument also gives Theorem C.

\vspace{7pt}

\noindent\emph{Proof of Theorem C}\quad\emph{Surjectivity}\quad We need to show that each $\psi \in \Z^p(G,A)$ is cohomolous to a cocycle taking values in $A_m$ for some finite $m$.  Let $\eps,\delta > 0$. Since $\psi$ is measurable there are some $m$ and $\psi_1 \in \C^p(G,A_m)$ such that $m_{G^p}(\{\psi \neq \psi_1\}) < \delta$.  Letting $\psi_2 := \psi - \psi_1$, it follows that $d\psi_2 = -d\psi_1$ takes values in $A_m$ and is $\rm{O}_p(\delta)$-small.  Using Lemma~\ref{lem:promoting-regularity-assumptions} and arguing as in the discrete-modules case of Theorem B, provided $\delta$ was chosen sufficiently small depending on $p$ and $\eps$ this implies that $d\psi_2 = d\psi_3$ for some $\eps$-small $\psi_3 \in \C^p(G,A_m)$.

Setting $\psi_4 := \psi_1 + \psi_3$, this is now a $p$-cocycle taking values in $A_m$ and we have that $\psi - \psi_4$ is still $(\delta + \eps)$-small. Therefore provided $\eps$ and $\delta$ were chosen sufficiently small, we may argue a second time as in the discrete-modules case of Theorem B to obtain that $\psi = \psi_4 + d\g$ for some $\g \in \C^{p-1}(G,A)$, as required.

\emph{Injectivity}\quad Now suppose that $\psi \in \Z^p(G,A_m)$ and that $\psi = d\g$ for some $\g \in \C^{p-1}(G,A)$; we must show that in fact $\g$ can be assumed to takes values in $A_{m'}$ for some finite $m' \geq m$.  However, letting $\bar{\g}:= \g + A_m$ be the image under the quotient map $A\to A/A_m$, this is now a Borel $(p-1)$-cocycle taking values in the discrete module $A/A_m$.  Therefore by the above-proved surjectivity there is some $\b\in \C^{p-2}(G,A)$ such that $\bar{\g} - d\bar{\b}$ takes values in $A_{m'}/A_m$ for some finite $m'$, and hence $\g - d\b$ takes values in $A_{m'}$.  Since $dd\b = 0$, we may always replace $\g$ with $\g - d\b$ in our representation of $\psi$ as a coboundary, and so complete the proof.  \qed

\vspace{7pt}

The following simple corollary of the above proof is a first step towards Theorem D, and will be useful shortly.

\begin{cor}\label{cor:disc-cohom-grps}
If $G$ is compact and $A$ is a discrete or toral $G$-module then the cohomology groups $\rmH_\m^p(G,A)$ are discrete and countable in their quotient topologies.
\end{cor}

\begin{proof}
The proof of Theorem B gives regularity under uniform approximations for these modules, and Lemma~\ref{lem:promoting-regularity-assumptions} converts this into the conditions needed for Lemma~\ref{lem:reg-implies-Haus}. \qed
\end{proof}

\subsection{Counterexamples among locally compact groups}\label{sec:loccpt}

We now offer some examples of the failure of analogs of Theorems B and C among locally compact groups.

Of course, any discrete group that is not compactly-generated such as the free group $\rm{F}_\infty$ can lead to terrible behaviour.  However, we will find that such phenomena occur also among groups of the form $\bbT^{\bbN}\times \G$ with $\G$ a finitely-generated discrete group.  These constructions rely on the following purely discrete result.

\begin{lem}\label{lem:locally-zero-cocycles}
There exists a finitely-generated group $\G$ with the property that for every finite subset $F \subset \G$ there is a non-trivial cocycle $\psi \in \Z^2(\G,\bbZ)$ with $\psi|_{F\times F} = 0$.
\end{lem}

\begin{proof}
Doubtless there are many ways to construct such a group $\G$; the method recorded here was indicated to us by Nicolas Monod.

Let $\rm{F}_2 = \langle s_1,s_2\rangle$ be a free group on two generators, and let
\[\Lambda := \langle s_2^ns_1s_2^{-n}:\ n\in \bbN\rangle \leq \rm{F}_2,\]
so $\Lambda \cong \rm{F}_\infty$.  Now let $\G := \rm{F}_2\ast_\Lambda\rm{F}_2$, so this is certainly finitely-generated. The Mayer-Vietoris sequence for homology (Corollary 7.7 in Brown~\cite{Bro82}) gives an exact sequence
\[\cdots \to \rmH_2(\Lambda,\bbZ)\to \rmH_2(\rm{F}_2,\bbZ)\oplus \rmH_2(\rm{F}_2,\bbZ)\to \rmH_2(\G,\bbZ) \to \rmH_1(\Lambda,\bbZ) \to \cdots.\]
On the one hand, $\rmH_1(\Lambda,\bbZ) \cong \rmH_1(\rm{F}_\infty,\bbZ)\cong \bbZ^{\oplus \infty}$ (direct sum), while on the other $\rmH_2(\rm{F}_2,\bbZ) = (0)$ and $\rmH_1(\rm{F}_2,\bbZ) \cong \bbZ^2$ (see Example 1 in Section 2.4 of~\cite{Bro82}), and so $\rmH_2(\G,\bbZ)$ is isomorphic to the kernel of some homomorphism $\bbZ^{\oplus \infty}\to \bbZ^4$ and so must also be isomorphic to $\bbZ^{\oplus\infty}$.

Next, a special case of the Universal Coefficient Formula (which we may apply in our setting since the Borel-cochains cohomology of a discrete group is simply its classical cohomology) gives an exact sequence
\[0\to \rm{Ext}^1_\bbZ(\rmH_1(\G,\bbZ),\bbZ)\to \rmH^2(\G,\bbZ)\to \rm{Hom}(\rmH_2(\G,\bbZ),\bbZ)\to 0.\]
Here we have $\rm{Hom}(\rmH_2(\G,\bbZ),\bbZ) \cong \rm{Hom}(\bbZ^{\oplus\infty},\bbZ) \cong \bbZ^\infty$ (direct product), and so it follows that $\rmH^2(\G,\bbZ) = \Z^2(\G,\bbZ)/\B^2(\G,\bbZ)$ admits an epimorphism onto $\bbZ^\infty$.  Therefore we can find some $\psi_1,\psi_2,\ldots \in \Z^2(\G,\bbZ)$ such that the sequence $\psi_i + \B^2(\G,\bbZ)$ is linearly independent in $\rmH^2(\G,\bbZ)$. 

Let $F\subset \G$ be finite and let $\rho:\Z^2(\G,\bbZ)\to \bbZ^{F\times F}:\psi\mapsto \psi|_{F\times F}$ be the corresponding restriction map. The sequence $\rho(\psi_1),\rho(\psi_2),\ldots \in \bbZ^{F\times F}$ cannot be linearly independent since the target group has rank $|F|^2 < \infty$.  Therefore there is some non-zero integer combination among the $\psi_i$ which is cohomologically nonzero but vanishes on $F\times F$, as required. \qed
\end{proof}

\begin{prop}\label{prop:noB}
Identify $\bbZ^m$ with $\bbZ^m\oplus \{0\}\subseteq \bbZ^n$ whenever $m\leq n$, and similarly identify these with initial-coordinate subgroups of $\bbZ^{\oplus\infty}$.  With this agreed, let $A_n := \bbZ^n$ so that $A_1 \subset A_2 \subset \cdots$ is an increasing sequence of discrete Abelian groups with direct limit $A := \bbZ^{\oplus\infty}$.

Then for $\G$ as in the preceding lemma and for $A$ endowed with the trivial $\G$-action, the direct limit of the inclusion homomorphisms $\rmH^p(\G,A_m)\to \rmH^p(\G,A)$ is not surjective.
\end{prop}

\begin{proof}
Let $F_1 \subset F_2 \subset \cdots$ be an exhaustion of $\G$ by finite subsets, and for each $n$ let $\psi_n\in \Z^2(\G,\bbZ)$ be a non-trivial cocycle that vanishes on $F_n\times F_n$.  Now the cocycle
\[(g,h)\mapsto (\psi_1(g,h),\psi_2(g,h),\ldots)\]
takes values in $A$ (since for any fixed $(g,h)$ the above sequence is eventually zero), but arguing coordinate-wise this map is clearly not obtained from an $A_n$-valued cocycle for any $n$. \qed
\end{proof}

We will now apply this proposition in a setting where the modules themselves are equal to the inner cohomology groups appearing in an LHS spectral sequence.  This will lead to an example of an inverse limit under which the $\bbT$-valued cohomology is not continuous.

\begin{prop}\label{prop:noA}
If $\G$ is as in the preceding lemma and for each $N$ we let $\pi_N:\bbT^\infty\times \G\to \bbT^N\times \G$ be the obvious coordinate projection, then the inverse limit of the inflation maps
\[\inf_{\pi_N}:\rmH_\m^3(\bbT^N\times \G,\bbT)\to \rmH_\m^3(\bbT^\infty\times \G,\bbT)\]
is not surjective.
\end{prop}

\begin{proof}
We view the groups $\bbT^N \times \Gamma$ and $\bbT^\infty \times \Gamma$ as group extensions by $\Gamma$ of the compact subgroups $\bbT^n$ and $\bbT^\infty$ respectively, and then use the LHS spectral sequences for these extensions to tease apart the structure of $\rmH^\ast_\m(\bbT^N \times \Gamma,\bbT)$ and $\rmH^\ast_\m(\bbT^\infty \times \Gamma,\bbT)$. (Recall that these spectral sequences for the theory $\rmH^\ast_\m$ are introduced and explained in~\cite{Moo76(gr-cohomIII)}.)  We therefore have a family of spectral sequences $E^{p,q,(N)}_r$ indexed by $N = 1,2,\ldots,\infty$, which for $r=2$ are given by
\[E^{p,q,(N)}_2 = \rmH^p(\G,\rmH_\m^q(\bbT^N,\bbT))\quad\quad\hbox{and}\quad\quad E^{p,q,(\infty)}_2 = \rmH^p(\G,\rmH_\m^q(\bbT^\infty,\bbT)),\]
and which abut to the canonical gradings of $\rmH_\m^\ast(G\times \G,\bbT)$ for $G = \bbT^N$ and $G = \bbT^\infty$ respectively.  In order that these spectral sequences exist at all we need to know that each of the cohomology groups $\rmH^q_\m(G,\bbT)$ is Hausdorff in its natural topology: this is just Corollary~\ref{cor:disc-cohom-grps}.

The groups $\rmH^q_\m(G,\bbT)$ can be computed exactly as follows. By Theorem A, for any compact group $G$ one has $\rmH_\m^q(G,\bbR) = (0)$ for $q\geq 1$, and so the long exact sequence arising from the presentation $\bbZ\into \bbR\onto \bbT$ collapses to give
\[\rmH_\m^q(G,\bbT) \cong \rmH^{q+1}_\m(G,\bbZ) \quad \forall q\geq 1.\]
Borrowing the special case of Theorem E for compact groups from Proposition~\ref{prop:cpt-and-disc} later in the paper (where it will be proved without reference to the current argument), this right-hand side is equal to $\rmH^{q+1}_{\rm{cs}}(G,\bbZ)$, and now combining this with the calculations in Theorem V.1.9 of Hofmann and Mostert~\cite{HofMos73} we have
\[\rmH_\m^q(G,\bbT)\cong\left\{\begin{array}{ll}\rm{Sym}^{(q+1)/2}\hat{G}&\quad\hbox{if}\ q\geq 1\ \hbox{odd}\\ (0)&\quad\hbox{if}\ q\geq 1\ \hbox{even}.\end{array}\right.\]

Hence for $G$ equal to either $\bbT^N$ and $\bbT^\infty$, the second tableau above appears as follows (continuing in both directions in the obvious way)

\vspace{7pt}

\begin{small}
\begin{center}
$\phantom{i}$\xymatrix{ \vdots & \vdots & \vdots & \vdots \\
(0) & (0) & (0) & (0) & \cdots  \\
\hat{G}\odot\hat{G} & \rmH^1(\G,\hat{G}\odot\hat{G}) & \rmH^2(\G,\hat{G}\odot\hat{G}) & \rmH^3(\G,\hat{G}\odot\hat{G}) & \cdots\\
(0) & (0) & (0) & (0) & \cdots\\
\hat{G} & \rmH^1(\G,\hat{G}) & \rmH^2(\G,\hat{G}) & \rmH^3(\G,\hat{G}) & \cdots\\
\bbT & \rmH^1(\G,\bbT) & \rmH^2(\G,\bbT) & \rmH^3(\G,\bbT) & \cdots\\}
\end{center}
\end{small}

\vspace{7pt}

The connecting maps here, whose homology will give rise to the third tableau, have bi-degree $(2,-1)$.  The diagonal $p+q = 3$ will abut to the grading of $\rmH_\m^3(G\times\G,\bbT)$, and it is clear that under the inflation maps $\inf_{\pi_N}^3$ this grading for $\rmH_\m^3(\bbT^N\times \G,\bbT)$ is carried into the grading for $\rmH_\m^3(\bbT^\infty\times \G,\bbT)$.

The entry $E^{3,0}_3$ in the third tableau obtained from the above is the kernel of the connecting map $\rmH^3(\G,\bbT)\to (0)$, hence stabilizes after the third tableau, and it also agrees with $E^{3,0,(N)}_3$ for every $N$.

On the other hand, the entry $E^{2,1}_3$ will be the kernel of the connecting map $\rmH^2(\G,\hat{G}) \to \rmH^4(\G,\bbT)$ (modulo the image of the map $(0) \to \rmH^2(\G,\hat{G})$, which of course is zero).  By construction of the spectral sequence, anything in the image of this connecting map vanishes when it is inflated to a class in $\rmH^4_\m(G\times \G, \bbT)$; but on the other hand, since $G\times \G$ is a direct product, the inflation maps $\rmH^\ast_\m(\G,\bbT) \to \rmH^\ast_\m(G\times \G,\bbT)$ are injective.  Therefore this connecting map is zero, and hence $E^{2,1}_3 = E^{2,1}_2 = \rmH^2(\G,\hat{G})$. Moreover, for $r\geq 3$ the differentials into and out of $E^{2,1}_r$ both vanish because their domains or ranges are outside the first quadrant $p \geq 0$, $q \geq 0$, so $E^{2,1}_r$ actually stabilizes to $\rmH^2(\G,\hat{G})$.

However, this group $\rmH^2(\G,\hat{G})$ is either $\rmH^2(\G,\hat{\bbT^N})$ or $\rmH^2(\G,\hat{\bbT^\infty})$ depending on $G$, and the previous proposition has shown that the latter is \emph{strictly} larger than the union of the images of the former under the inclusion maps $\hat{\bbT^N}\circ\pi_N\subset \hat{\bbT^\infty}$.

It follows that $E^{2,1,(\infty)}_r$ contains entries that are not lifted from any $E^{2,1,(N)}_r$, and these elements persist in the resulting grading of $\rmH_\m^3(\bbT^\infty \times \G,\bbT)$.  This proves the asserted non-surjectivity. \qed
\end{proof}

\vspace{7pt}

\noindent\emph{Remark}\quad It is worth noting the r\^ole played by Proposition~\ref{prop:noB} in Proposition~\ref{prop:noA}.  It seems that this could be reversed: if a group $\G$ can be shown to satisfy the conclusion of Theorem C, then any extension of a compact group by $\G$ will admit an analog of Theorem B.  This is because the LHS spectral sequence together with Theorem B for compact groups effectively convert the inverse limit of the base groups into a direct limit of cohomology groups within each entry of the second tableau, to which the conclusion of Theorem C can be applied.  We will not give a careful proof of this here. \fin

\vspace{7pt}

We suspect that examples also exist in which the direct limit of inflation maps is not injective, but have none to hand.  On the other hand, from the automatic continuity of crossed homomorphisms it is easily seen that
\[\rmH_\m^1(G,A) = \lim_{n\leftarrow}\rmH_\m^1(G_n,A)\]
under the direct limit of inflation maps whenever $G_n = G/K_n$ for some decreasing sequence of compact normal subgroups $K_n \unlhd G$ with $\bigcap_{n\geq 1}K_n = \{1\}$.

\begin{ques}\label{ques:thm-B-fail}
Are there examples of locally compact, second countable, compactly-generated groups $G$ for which Theorem B fails in degree $2$ for toral or discrete targets? \fin
\end{ques}

One might hope to recover a version of Theorems B or C for locally compact groups subject to some additional assumptions.

\begin{ques}\label{ques:B-or-C-with-extra-assumptions}
Does Theorem B or C hold if $G$ is locally compact, compactly-generated and second countable, $G_n = G/K_n$ for some decreasing sequence of compact subgroups $K_n \unlhd G$, and if
\begin{itemize}
\item $G$ is connected, or
\item the first quotient $G/K_1$ is a discrete $\rm{FP}_\infty$ group (see Chapter VIII of Brown~\cite{Bro82})? \fin
\end{itemize}
\end{ques}

\section{Agreement with other cohomology theories}\label{sec:other-theories}

We have already examined the continuous-cochains theory $\rmH^\ast_\rm{cts}$ in the proof Theorem A above.  We now turn to the comparison of $\rmH^\ast_\m$ with $\rmH^\ast_\rm{ss}$ and $\rmH^\ast_\rm{cs}$.

Recall that Wigner defines $\rmH^\ast_\rm{ss}$ (he denotes it $\hat{\rmH}^\ast$) using sheaves on semi-simplicial spaces as follows.  Given $G$ and $A$, let $G^\bullet$ be the associated semi-simplicial space given by the Cartesian powers of $G$ and the coordinate-deletion maps, and let $\cal{A}^\bullet$ be the semi-simplicial sheaf of germs of continuous maps $G^\bullet \to A$.  By forming the (second) canonical resolution of each $\cal{A}^n$ one obtains a double complex, and the theory $\rmH^\ast_\rm{ss}$ is the cohomology of the resulting total complex, which may be computed by the associated spectral sequence.  This construction is given in detail in Section 3 of~\cite{Wig73}, and also nicely explained in Section 2 of Lichtenbaum~\cite{Lic09}.  Importantly, the definition of this theory works just the same on the larger category of all semi-simplicial Abelian sheaves on $G^\bullet$, where it can be seen as the semi-simplicial version of the usual construction of derived functors using injective resolutions of sheaves.

On the other hand, letting
\[G\into E_G \stackrel{\pi}{\onto} B_G\]
be a choice of classifying principal $G$-bundle and classifying space, one can let $\cal{A}$ be the sheaf on $B_G$ defined by
\[\cal{A}(U) = \rm{Map}(\pi^{-1}(U),A)^G,\]
where $\rm{Map}$ denotes the set of continuous maps.  In case $A$ is discrete the resulting sheaf cohomology will be denoted by $\rmH^\ast_\rm{cs}(G,A)$.  The choice of $B_G$ is unique only up to homotopy, but we know that the sheaf cohomology is homotopy-invariant in case $A$ is discrete (see, for instance, Chapter 5 of Schapira's lecture notes~\cite{Sch07}). 

More na\"\i vely, one could consider a more classical cohomology theory (such as singular or \v{C}ech) of $B_G$ with coefficients in the Polish group $A$.  This coincides with the sheaf-theoretic cohomology in case $A$ is discrete, but in general it ignores the topology of $A$, and so it seems inappropriate for comparison with $\rmH_\m^\ast(G,A)$.  Indeed, by rights it should compute $\rmH_\m^\ast(G,A_\rm{d})$, where $A_\rm{d}$ is the group $A$ with its discrete topology, although this may not be Polish.  In Section~\ref{sec:further-discuss} we will give an example offering different points of view on this shortcoming.

In this section we will first recall various comparison maps between these theories, and will then establish some cases in which those maps define isomorphisms.  The comparison maps are shown in the commutative diagram below. As before, we always assume that $G$ is locally compact and second countable and that $A$ is Polish, but in this diagram a dashed arrow indicates a map whose definition requires some additional assumptions on $G$ or $A$.

\begin{center}
$\phantom{i}$\xymatrix{\rmH^\ast_\rm{cts}\ar^\phi[rr]\ar_{\iota_1}[drr] & & \rmH^\ast_{\rm{ss}}\ar@{-->}^{\psi}[rr]\ar^\a[d] & & \rmH^\ast_\rm{cs}\ar@{-->}^{\iota_2}[dll]\\
&& \rmH_\m^\ast
}
\end{center}

\subsection{Constructing the comparison maps}

Consider first the arrows emanating from $\rmH^\ast_\rm{cts}$.  Of course $\iota_1$ arises from the obvious injection of continuous cochains.  The map $\phi$ comes from the edge map of the spectral sequence for $\rmH^\ast_\rm{ss}$, as explained in Section 3 of Wigner~\cite{Wig73}.  Segal constructs a similar map $\rmH^\ast_\rm{cts}\to \rmH^\ast_\rm{Seg}$ for his category of modules in Section 3 of~\cite{Seg70}.

The map $\a$ is constructed by induction on degree using dimension-shifting and the universality property of $\rmH^\ast_\rm{ss}(G,\cdot)$ on its category of definition (semi-simplicial sheaves of Abelian groups on $G^\bullet$). An easy check shows that $\rmH^0_\rm{ss}(G,A) \cong A^G \cong \rmH_\m^0(G,A)$, giving a natural isomorphism in degree $0$.  Now suppose we have constructed homomorphisms $\rmH^r_\rm{ss}(G,\,\cdot\,)\to \rmH_\m^r(G,\,\cdot\,)$ for all $r \leq p$, and consider $\rmH^{p+1}_\rm{ss}(G,A)$.  This is defined using the sheaf $\cal{A}^{p+1}$ of germs of continuous functions $G^{p+1} \to A$.

Any element of $\Z^{p+1}(G,A)$ is effaced by the inclusion $A \into \C(G,A)$, but it may alternatively by effaced by the further inclusion into $E(A) := \C_\rm{cts}(G,\C(G,A))$.  Let $F(A) := E(A)/A$, and let $\cal{E}(E)$ be the semi-simplicial sheaf corresponding to $E(A)$.  Then $\A$ and $\E(A)$ fit into a short exact sequence
\[\A\into \E(A)\onto \t{\F}(A)\]
where $\t{\F}(A)$ is the quotient sheaf.  Importantly, this is not in general equal to the sheaf $\F(A)$ corresponding to $F(A)$: $\t{\F}(A)$ corresponds to those members of $\F(A)$ that locally admit continuous lifts to members of $\E(A)$.  Thus there is a canonical inclusion $\t{\F}(A)\to \F(A)$.

Now, dimension-shifting within the categories appropriate to $\rmH_\m^\ast$ and $\rmH^\ast_\rm{ss}$ gives canonical switchback isomorphisms
\[\rmH_\m^{p+1}(G,A)\cong \rmH_\m^p(G,F(A))\quad\hbox{and}\quad \rmH^{p+1}_\rm{ss}(G,A)\cong \rmH^p_\rm{ss}(G,\t{\F}(A)).\]

The inclusion $\t{\F}(A)\to \F(A)$ gives a map $\rmH^p_\rm{ss}(G,\t{\F}(A))\to \rmH^p_\rm{ss}(G,\F(A))$, and by induction we have already constructed a map from this to $\rmH^p_\m(G,F(A))$.  Composing these with the above isomorphisms defines the map $\rmH^{p+1}_\rm{ss}(G,A) \to \rmH_\m^{p+1}(G,A)$.

In general there is no reverse construction from $\rmH_\m^\ast$ to $\rmH^\ast_\rm{ss}$, because $\t{\F}(A)$ is included in $\F(A)$ and not the other way around.  Indeed, the possible disagreement  $\t{\F}\neq \F$ means that $\rmH^\ast_\rm{ss}$ may not have long exact sequences when restricted to the category of Polish modules, and so the universality of $\rmH_\m^\ast$ within this category is no help in making such a comparison.

If, however, some conditions are imposed to ensure that any continuous function from an open subset of some $G^n$ to $F(A)$ can be lifted locally around any point to a continuous function to $E(A)$, then we obtain an equality of sheaves $\F = \t{\F}$, and so the above construction actually defines an isomorphism of cohomology theories. Such an argument appears in Wigner~\cite{Wig73}, where the special subcategory of Polish $G$-modules having `property F' is isolated and shown to give isomorphisms $\rmH_\m^\ast(G,A)\cong \rmH^\ast_\rm{ss}(G,A)$ when $G$ is finite-dimensional and $A$ has property F.  Property F is precisely a condition on the local lifting of continuous maps from locally Euclidean spaces, and this gives rise to an isomorphism of the necessary sheaves as above.

(This reasoning is also responsible for the agreement of $\rmH^\ast_\rm{ss}$ with $\rmH^\ast_\rm{Seg}$ on Segal's category of modules: that category consists of modules that are locally contractible k-spaces, and one can show that in this category cocycles may be effaced by inclusions into cohomologically-zero modules that admit local cross-sections.  Put concisely, this amounts to an exact embedding of Segal's category, where all exact sequences have local cross-sections, into the semi-simplicial category, and the cohomology functors $\rmH^\ast_\rm{Seg}$ and $\rmH^\ast_\rm{ss}$ are isomorphic on this subcategory.)

In case $A$ is discrete, the map $\psi$ is constructed by Wigner in~\cite{Wig73} by an entry-wise comparison of the spectral sequence computing $\rmH^\ast_\rm{ss}(G,A)$ with a spectral sequence for $\rmH^\ast(B_G,A)$.  A similar argument appears in Section 3 of Segal~\cite{Seg70}, and given that $\rmH^\ast_\rm{ss}$ and $\rmH^\ast_\rm{Seg}$ agree for discrete $A$ these constructions are easily seen to coincide. These maps will be less important in the sequel, so we simply refer the reader to those references.

Finally, a map $\iota_2$ may be defined directly for discrete target modules $A$.  Its construction is very elementary compared with the above sheaf-theoretic arguments, but it is not formally needed for any of our isomorphism results and requires more preparation, so it will be deferred to Section~\ref{sec:further-discuss} below.

\subsection{First new instances of isomorphism}

Theorem E is our main result giving conditions for these comparison maps to be isomorphisms, and its proof will not appear until Section~\ref{sec:compar-complete}
as it will require some other developments of the next section.  Here we will recall or prove a few simpler results relating these various theories.

The principal cases of isomorphism obtained either here or in previous papers are listed in the table below.

\vspace{7pt}

\begin{center}
\begin{tabular}[center]{p{4cm}ccc}
\hline
\quad\quad\quad\quad\quad\quad\quad $\rmH_\m^\ast = \ldots$ & $\rmH^\ast_\rm{cts}$ & $\rmH^\ast_\rm{ss}$ & $\rmH^\ast_\rm{cs}$\\
\hline
$A$ Fr\'echet & $\surd$ & $\surd$ & $\times$\\
\hline
$G$ totally disconnected & $\surd$ & $\surd$ & $\times$\\
\hline
$A$ discrete & $\times$ & $\surd$ & $\surd$\\
\hline
$A$ locally compact and locally contractible & $\times$ & $\surd$ & $\times$\\
\hline
$G$ finite-dimensional, $A$ has property F & $\times$ & $\surd$ & $\times$\\
\hline
\end{tabular}
\end{center}

\vspace{7pt}

In this table a tick indicates that the assumptions on $(G,A)$ listed for that row suffice to imply equality between $\rmH_\m^\ast$ and another theory, and a cross indicates that they do not suffice (although of course equality may still hold under more restrictive assumptions).

The first row of this table is given by Theorem A and the corresponding results for $\rmH^\ast_\rm{ss}$, which follow from the degeneration of the second tableau of the spectral sequence that defines it; this latter is explained for the proof of Theorem 3 in~\cite{Wig73}, or also for Lemma 2.5 and Corollary 2.6 in Lichtenbaum~\cite{Lic09}.  There results together amount to the Fr\'echet-module case of Theorem E.

The isomorphism $\rmH_\m^\ast(G,A) \cong \rmH^\ast_\rm{cts}(G,A)$ for totally disconnected $G$ appears in~\cite{Wig73} as Theorem 1.  That argument is analogous to our proof of Theorem A above.  The point is that if $G$ is totally disconnected then any continuous map from $G^n$ to a Polish $G$-module $A$ can be continuously lifted through any surjection from another Polish $G$-module onto $A$, by Michael's selection theorems from~\cite{Mic56,Mic57}.  As a result the functor $\rmH^\ast_\rm{cts}(G,\cdot)$ gives a long exact sequence for any short exact sequence in $\sfP(G)$ and so agrees with $\rmH_\m^\ast(G,\cdot)$ by Buchsbaum's criterion. The demonstration that $\a$ is an isomorphism in this case follows the same principle (and also bears comparison with Wigner's proof of his Theorem 2), although here effaceability holds in general and the point is to show that $\rmH^\ast_\rm{ss}(G,\cdot)$ is still a cohomological functor when restricted to the category of $G$-modules. Another appeal to the Michael selection theorem shows that for an exact sequence as in~(\ref{eq:sh-ex}) the corresponding sequence of semi-simplicial sheaves on $G^\bullet$,
\[0 \to \cal{A}^\bullet \to \cal{B}^\bullet \to \cal{C}^\bullet \to 0,\]
is exact as $G^n$ as totally disconnected for all $n$. Hence $\rmH^\ast_\rm{ss}(G,\cdot)$ is a cohomological functor, and since we know it is also effaceable, it agrees with the other two by Buchsbaum's criterion.

The fifth row of the table above is the conjunction of Theorems 2 and 4 from Wigner~\cite{Wig73}, and we refer the reader there for the proof (and also a precise explanation of Property F). Let us recall for reference that it implies, in particular, isomorphisms
\[\rmH_\m^\ast(G,A) \cong \rmH^\ast_\rm{ss}(G,A) \cong \rmH^\ast_\rm{cs}(G,A)\]
whenever $A$ is discrete and $G$ is finite-dimensional.

The third and fourth rows correspond to the conclusion of Theorem E for locally compact and locally contractible modules.  Their proofs will be completed in Section~\ref{sec:compar-complete}, but before leaving this section we will establish some preliminary special cases of them that will be needed in the mean time.

The first of these is a direct consequence of Theorems B and C.

\begin{prop}\label{prop:cpt-and-disc}
If $G$ is compact and $A$ is discrete then $\rmH_\m^\ast(G,A) \cong \rmH^\ast_\rm{ss}(G,A) \cong \rmH^\ast_\rm{cs}(G,A)$.
\end{prop}

\begin{proof}
This is already known for compact Lie groups by Wigner's results for finite-dimensional $G$ (see also Section 2 of~\cite{Lic09}).  On the other hand, all three theories have the continuity properties of our Theorems B and C when the target $A$ is discrete: this follows from Proposition III-1.11 of~\cite{HofMos73} for $\rmH^\ast_\rm{cs}$ and Proposition 8.1 and Corollary 7 in~\cite{Fla08} for $\rmH^\ast_\rm{ss}$.  The result now follows for general compact second countable $G$ by ascending some inverse sequence of Lie groups that converges to $G$ (obtained, for example, from the Peter-Weyl Theorem). \qed
\end{proof}

The next proposition gives the special case of Theorem E for connected groups $G$.  It will be needed for the proof of the full versions of those theorems in Section~\ref{sec:compar-complete} below.  The proof will make use of the LHS spectral sequence, and also the remarkable structural result from works of Gleason, Montgomery and Zippin and of Yamabe that any connected locally compact group $G$ is an extension of a compact group by a Lie group (see Theorem 4.6 in Montgomery and Zippin~\cite{MonZip55}):
\begin{eqnarray}\label{eq:H5}
\begin{array}{ccccccccc}1 & \to  & H & \to & G & \to &  G/H & \to & 1\\ &   & \hbox{\small{compact}}&  &  &  &  \hbox{\small{Lie}} &  & \end{array}
\end{eqnarray}

Of course, cohomology for Lie groups is relatively much better understood than for general locally compact groups, especially in the setting of Fr\'echet targets where a host of additional techniques and relations to Lie algebra cohomology are known.  A thorough reference for these matters is Borel and Wallach's book~\cite{BorWal00}.

\begin{prop}\label{prop:ThmE-connected}
If $G$ is a connected locally compact group and $A$ is discrete or Euclidean then the map $\rmH^\ast_\rm{ss}(G,A) \stackrel{\a}{\to} \rmH_\m^\ast(G,A)$ is an isomorphism.
\end{prop}

\begin{proof}
If $G$ is connected then the Gleason-Montgomery-Zippin Theorem gives a presentation as in~(\ref{eq:H5}).  By Corollary~\ref{cor:disc-cohom-grps} or Theorem A, the groups $\rmH_\m^\ast(H,A)$ are discrete in their quotient topologies, and hence by Theorem 9 from~\cite{Moo76(gr-cohomIII)} there is an abutting spectral sequence
\[E_2^{pq} \cong \rmH_\m^p(G/H,\rmH_\m^q(H,A)) \Longrightarrow \rmH_\m^{p+q}(G,A).\]

On the other hand, there is also such a spectral sequence for $\rmH^\ast_\rm{ss}$ because the quotient $G\to G/H$ has local cross-sections (since the quotient is a Lie group), and so the proof is completed by observing that
\[\rmH_\m^q(H,A)\cong \rmH^q_\rm{ss}(H,A)\]
by Proposition~\ref{prop:cpt-and-disc} and
\[\rmH_\m^q(G/H,B) \cong \rmH^p_\rm{ss}(G/H,B)\]
for any discrete module $B$ by Wigner's Theorem 2 concerning finite-dimensional acting groups. \qed
\end{proof}

Similarly to the discussion following the proof of Theorem A, it is unclear whether the theories $\rmH_\m^\ast$ and $\rmH^\ast_\rm{ss}$ must agree for more general contractible targets.

\begin{ques}\label{ques:contractible-disagree-with-ss}
Are there a connected locally compact group $G$ and a non-locally-convex F-space $G$-module $A$, or more generally a contractible Polish module, for which
\[\rmH_\m^p(G,A) \not\cong \rmH^p_\rm{ss}(G,A)\]
for some $p$?
\end{ques}

The completion of Theorem E will require further use of the LHS spectral sequence, but we will not be able to justify that appeal until we have some results guaranteeing that the quotient topologies on some relevant cohomology groups are Hausdorff. These will be obtained in the next section.

Before leaving this section, it seems worth offering some further remarks on possible strategies for analyzing the comparison map between $\rmH_\m^\ast$ and $\rmH^\ast_\rm{ss}$.  For the purpose of this discussion it is simplest to focus on cases in which $\rmH^\ast_\rm{Seg}$ is defined, so that we may discuss this in place of $\rmH^\ast_\rm{ss}$.

Consider our proof of Theorem A.  It was built on a result promising measurable cohomology-class representatives $G^p\to A$ that have some additional regularity (local total boundedness).  From here a simple smoothing argument gave a proof that $\rmH_\m^\ast$ admits the effacement of cocycles within the category of Fr\'echet modules, just as $\rmH^\ast_\rm{cts}$ does.  This means that $\rmH_\m^\ast$ and $\rmH^\ast_\rm{cts}$ define two connected cohomological sequences of functors which both admit effacement within the \emph{same} category, so since they agree in degree zero Buchsbaum's criterion implies that they agree in all degrees.

This is a simple and attractive strategy for comparing $\rmH_\m^\ast$ with another theory: first one proves that measurable cocycle representatives can be found with some additional regularity, and then one deduces that $\rmH_\m^\ast$ has the same universality properties as the other theory within some smaller common category of modules where both are defined.  However, it seems difficult to implement this strategy for comparing $\rmH_\m^\ast$ with $\rmH^\ast_\rm{ss}$ or $\rmH^\ast_\rm{Seg}$. (The instances of agreement proved in~\cite{Wig73} do rely on Buchsbaum's criterion, but not on any regularizing of measurable cocycles.)  Proposition~\ref{prop:basic-smoothness} gives some ability to improve the regularity of cocycles in great generality, but it is not clear that this improvement is good enough to make the link to $\rmH^\ast_\rm{ss}$ or $\rmH^\ast_\rm{Seg}$.  For this reason our final proof of Theorem E will instead need very detailed information about the possible structure of the underlying group $G$ in the form of the Gleason-Montgomery-Zippin Theorem, applied via the LHS spectral sequence.

To explain this difficulty further let us focus on $\rmH^\ast_\rm{Seg}$.  That theory is defined for locally contractible k-space modules, and with exact sequences of such modules being required to allow local cross-sections.  To prove effacement in this category, Segal first embeds such a module $A$ into a larger module $E_\rm{Seg}A$ (which may be identified as the subgroup of finite-valued step functions in $\C([0,1],A)$, although Segal works mostly with a different description coming from~\cite{Seg68}), which is contractible and in which the embedded copy of $A$ has a local cross-section; and thence into the module $\C_\rm{cts}(G,E_\rm{Seg}A)$ which is `soft' in the sense of his theory and hence cohomologically zero.  An important part of Segal's work is to prove that $A$ has a local cross-section in $E_\rm{Seg}A$ (Proposition A.1 in~\cite{Seg70}).

Unfortunately, $E_\rm{Seg}A$ is not Polish. In order to connect $\rmH^\ast_\rm{Seg}(G,A)$ with $\rmH_\m^\ast(G,A)$ for a given module $A$, we would want to efface $A$-valued cocycles by embedding into a \emph{Polish} module which is contractible and within which the copy of $A$ has a local cross-section.  The canonical choice of effacing embedding is $A \into \C(G,A)$, whose target is clearly contractible when $G$ is non-discrete, but for which we do not know whether there is a local cross-section.  Proposition~\ref{prop:basic-smoothness} enables some improvement here: effacement obtains if instead we embed $A$ into the smaller module of locally totally bounded Borel maps $G\to A$, so for instance we could embed into any of the Polish modules $L^p_\rm{loc}(G,A)$ of functions $f:G\to A$ for which the distance function $g\mapsto \rho(0,f(g))$ is locally $p$-integrable.

We have not been able to find any category $\sf{Cat}$ of Polish modules $A$ other than Fr\'echet modules for which such an inclusion $A \into L^p_\rm{loc}(G,A)$ and the resulting quotient $L^p_\rm{loc}(G,A)/A$ are both still objects of $\sf{Cat}$ and such that we can prove that this inclusion has a local cross-section.  Indeed, in most cases the status of a local cross-section seems to be unknown.

For example, if $A = \bbZ$ then a first dimension-shifting would involve $\bbZ\into L^p_\rm{loc}(G,\bbZ)$, which easily does admit a local cross-section, but then the next dimension-shifting would require
\begin{multline*}
L^p_\rm{loc}(G,\bbZ)/\bbZ \into L^p_\rm{loc}\big(G,L^p_\rm{loc}(G,\bbZ)/\bbZ\big)/(L^p_\rm{loc}(G,\bbZ)/\bbZ)\\ \cong L^p_\rm{loc}(G\times G,\bbZ)/(\pi_1^\ast L^p_\rm{loc}(G,\bbZ) + \pi_2^\ast L^p_\rm{loc}(G,\bbZ))
\end{multline*}
where $\pi_i:G\times G\to G$ are the two coordinate projections for $i=1,2$.  Clearly it is unimportant that the underlying space here is $G$; the key issue is captured by the following relative of Question~\ref{ques:R-in-L0}.

\begin{ques}\label{ques:subgp-of-L1(sq)}
Does the closed, contractible subgroup
\[\pi_1^\ast L^1([0,1],\bbZ) + \pi_2^\ast L^1([0,1],\bbZ)\]
admit a local cross-section in $L^1([0,1]^2,\bbZ)$?
\end{ques}

In the case of $\bbZ$, we could alternatively first embed into $\bbR$ to obtain the quotient $\bbT$, but then another dimension-shifting seems to require embedding this into some group such as $\C(G,\bbT)$.  This second embedding also admits local cross-sections (another simple exercise), but once again we do not know whether the \emph{next} dimension-shifting does so.

So the new possibilities for effacement of measurable cocycle-classes that are enabled by Proposition~\ref{prop:basic-smoothness} may not be strong enough to reach the defining conditions of $\rmH^\ast_\rm{Seg}$.  Hence the comparison must be made differently.  For example, Wigner's proof that $\rmH_\m^\ast(G,A)$ and $\rmH^\ast_\rm{ss}(G,A)$ agree for finite-dimensional $G$ and a large subcategory of Polish modules $A$ (including discrete modules) is based on the realization that for such a group $G$, the manipulations underlying the construction of $\rmH^\ast_\rm{ss}$ do not require local cross-sections for exact sequences of possibly-infinite-dimensional modules, but only the weaker assumption that these exact sequences have the homotopy lifting property for finite-dimensional spaces: this is Wigner's `Property F'.  See the proof of Theorem 2 in~\cite{Wig73}; a simple exercise shows that for finite-dimensional $G$ and $A$ having this property, one could similarly make a direct comparison between $\rmH^\ast_\m$ and $\rmH^\ast_\rm{Seg}$.  For many choices of $A$ this more modest lifting property is easy to verify for the exact sequence $A \to \C([0,1],A)\to \C([0,1],A)/A$, so the extra kinds of regularity enabled by Proposition~\ref{prop:basic-smoothness} are not needed for Wigner's development.

Nevertheless, once one has the isomorphism $\rmH_\m^\ast(G,\bbZ)\cong \rmH^\ast_\rm{Seg}(G,\bbZ)$ that will come from Theorem E, in some sense it itself asserts some regularity of cocycle-representatives.  If one implements this isomorphism on the level of cochains, it means that classes in $\rmH_\m^p(G,\bbZ)$ have representatives that are coboundaries of Borel maps
\[G^{p-1}\to \C_\rm{cts}(G,E_\rm{Seg}\bbZ),\]
where we recall again that $E_\rm{Seg}\bbZ$ may be identified with the subgroup of $\C([0,1],\bbZ)$ consisting of the finite-valued step functions. Such a coboundary $\s$ enjoys a (rather subtle) strengthening of the kind of regularity promised by Proposition~\ref{prop:basic-smoothness}: it can have only very `simple' discontinuities as a function on $G^p$.  Some results along these lines can be proved directly using a similar argument to Proposition~\ref{prop:basic-smoothness}, as will be discussed at the end of Section~\ref{sec:further-discuss} below, but it is not yet clear how those results can be related to the comparison theorems.

\section{Some consequences for the quotient topologies of cohomology groups}\label{sec:cohom-top}

The four parts of Theorem D are proved in this section as separated propositions and a corollary.  These results all use the assumption that $G$ is \textbf{almost connected}: that is, that $G/G_0$ is compact, where $G_0 \unlhd G$ is the identity component.  In a sense, this situation is opposed to the case of discrete groups, and it is this special structure that makes the following results possible.

\begin{prop}\label{prop:nice-top}
If $G$ is almost connected and $A$ is a Euclidean space then each $\B^p(G,A)$ is closed in $\Z^p(G,A)$ and $\rmH_\m^p(G,A)$ is a Euclidean space.
\end{prop}

In the nomenclature of Hochschild and Mostow~\cite{HocMos62}, this asserts that almost connected $G$ are of `finite homology type'. Their definition concerned $\rmH^\ast_\rm{cts}$, but by Theorem A that agrees with $\rmH^\ast_\m$ here.  The proof that follows bears comparison with the discussion following their Theorem 7.1.

\begin{proof}
First observe that it suffices to prove this in case $G$ is connected.  Indeed, letting $G_0 \unlhd G$ be the identity component, if we know that the groups $\rmH^p_\m(G_0,A)$ are Hausdorff then the LHS spectral sequence gives the abutment
\[E_2^{pq} \cong \rmH_\m^p(G/G_0,\rmH^q_\m(G_0,A)) \Longrightarrow \rmH_\m^{p+q}(G,A).\]
Since $G/G_0$ is compact, if each $\rmH^q_\m(G_0,A)$ is Euclidean then all of these entries vanish for $p > 0$ by Theorem A, and so the result for $G$ follows from that for $G_0$.

If $G$ is connected, then by the Gleason-Montgomery-Zippin Theorem it has a compact normal subgroup $H$ such that $G/H$ is Lie, and now another instance of the LHS spectral sequence gives
\[E_2^{pq} \cong \rmH_\m^p(G/H,\rmH^q_\m(H,A)) \Longrightarrow \rmH_\m^{p+q}(G,A).\]
This time, Theorem A gives $E_2^{pq} = (0)$ whenever $q > 0$, implying that the cohomology $\rmH^\ast_\m(G,A)$ is the isomorphic image of $\rmH^\ast_\m(G/H,A)$ under inflation.  We may therefore assume that $G$ is a Lie group.  Finally, Theorem A also gives that $\rmH_\m^\ast(G,A) = \rmH^\ast_\rm{cts}(G,A)$. In that case, by the results of Hochschild and Mostow~\cite{HocMos62}, $\rmH^\ast_\rm{cts}$ is given by Lie algebra cohomology, which is automatically finite dimensional.

So we know that $\B^p(G,A)$ is a topological vector space of finite codimension in $\Z^p(G,A)$ for each $p$. Since it is the image of a continuous surjection from the Polish group $\C^{p-1}(G,A)$, it is an analytic subset of $\Z^p(G,A)$. Now choose a finite set of vectors $\{e_i\}$ in $\Z^p(G,A)$ whose images form a basis in $\rmH_\m^p(G,A)$.  Then the map $f: (a_1,\ldots,a_n) \mapsto \sum_i a_i e_i$ from $\bbR^n$ into $\Z^p(G,A)$ is a surjective Borel map, hence a Borel isomorphism onto its image.

Therefore if $X$ is any Borel set in $\bbR^n$,  $f(X)$ is a Borel set in $\Z^p(G,A)$. Then the sum $\B^p(G,A) + f(X)$ is a continuous image of the analytic set $\B^p(G,A)\times f(X)$ and is therefore still analytic. Applying the same reasoning to $\bbR^n\setminus X$, one obtains that $\B^p(G,A) + f(X)$ is also co-analytic, and hence it is Borel.  This in turn says that the homomorphism from $\Z^p(G,A)$ to $\bbR^n$ that is a left-inverse to $f$ and has kernel $\B^p(G,A)$ is a Borel map from a Polish group to a Polish group and is hence continuous.  Therefore $\B^p(G,A)$ is closed and this map defines a homeomorphism from $\rmH_\m^p(G,A)$ to $\bbR^n$. \qed
\end{proof}

\noindent\emph{Remark}\quad Well-known examples show that the above may fail without the assumption that $A$ is finite-dimensional.  Indeed, any amenable group $G$ (connected, discrete or otherwise) admits an affine isometric action on a Hilbert space $\frH$ with no fixed points, but with approximate fixed points.  This affine isometric action may be specified in terms of a linear isometric action $\pi:G\actson \frH$, and a $1$-cocycle $b:G\to\frH$ for $\pi$ that specifies the translational part of the action; and having approximate but not exact fixed points corresponds to $b$ being an approximate coboundary but not an exact coboundary. See, for instance, de Cornulier, Tessera and Valette~\cite{deCTesVal07}. \fin

\begin{prop}\label{prop:top-disc-target}
If $G$ is almost connected and $A$ is discrete then each $\B^p(G,A)$ is closed in $\Z^p(G,A)$ and $\rmH_\m^p(G,A)$ is countable and discrete in its quotient topology.
\end{prop}

\begin{proof}
First assume that $G$ is connected and let $K$ be the maximal compact subgroup of $G$. Then as topological spaces one has $G = K \times V$ with $V$ a finite-dimensional real vector space. For the theory $\rmH^\ast_\rm{ss}$ the defining spectral sequence has first tableau
\[E_1^{pq} := \rmH^q_\rm{sheaf}(G^p, \cal{A}),\]
where $\cal{A}$ is the locally constant sheaf on $G^p$ corresponding to $A$. But topologically $G^p = K^p \times V^p$ and this is homotopy equivalent to $K^p$, so the sheaf cohomology on $G^p$ is the same as on $K^p$ by the homotopy principle for locally constant sheaves (see Chapter 5 of Schapira~\cite{Sch07}). Therefore the $E_1$-terms of this spectral sequence for $\rmH^\ast_\rm{ss}(G,A)$ coincide with those for $\rmH^\ast_\rm{ss}(K,A)$, and so the restriction map from $G$ to $K$ induces an isomorphism in cohomology for $\rmH^\ast_\rm{ss}$. By Proposition~\ref{prop:ThmE-connected} this holds also for $\rmH_\m^\ast$, and so Corollary~\ref{cor:disc-cohom-grps} implies that the groups $\rmH^p_\m(G,A)$ are countable.

We now use the following general result: if $B$ is any Polish group, $C$ is a normal subgroup which is an analytic subset, and $B/C$ is countable, then $C$ is open (and hence closed) in $B$ and the quotient group $B/C$ has the discrete topology. To see this, observe that by translation any coset of $C$ in $B$ is also analytic, and hence by countability any union of cosets of $C$ is analytic. Hence any union of cosets is both analytic and co-analytic, and then by the separation theorem any union of cosets is a Borel set. Therefore the homomorphism from $B$ to $B/C$ is a Borel homomorphism, and hence continuous, when the latter is given its \emph{discrete} topology. Therefore $C$ is open and also closed and the quotient topology on $B/C$ is itself discrete. Applying this with $B = \Z^p(G,A)$ and $C = \B^p(G,A)$, where the latter is analytic as a continuous image of $\C^{p-1}(G,A)$, it follows that $\rmH^p_\m$ is discrete in its quotient topology.

Now let $G$ be almost connected and let $G_0$ be its identity component. Since the groups $\rmH_\m^p(G_0,A)$ are discrete the LHS spectral sequence is available to give
\[E_2^{pq} = \rmH_\m^p(G/G_0, \rmH_\m^q(G_0, A)) \Longrightarrow \rmH_\m^{p+q}(G,A);\]
and since $G/G_0$ is totally disconnected we can can use $\rmH^\ast_\rm{cts}(G/G_0,\rmH_\m^\ast(G_0, A))$ in place of $\rmH_m^\ast(G/G_0,\rmH_\m^q(G_0,A))$.  Since $G/G_0$ is compact the group of cochains $\C^p_\rm{cts}(G/G_0,B)$ is countable for discrete countable $B$, and so the cohomology groups $\rmH^\ast_\rm{cts}(G/G_0,\rmH_\m^\ast(G_0, A))$ are also countable. Since in this case all groups in $E_2^{\ast\ast}$ are countable, the same follows for $\rmH_\m^\ast(G,A)$, and as before this implies a discrete quotient topology. \qed
\end{proof}

\vspace{7pt}

Proposition~\ref{prop:top-disc-target} illustrates the value of the comparison theorems between $\rmH_\m^\ast$ and $\rmH^\ast_\rm{ss}$: the fact that $\rmH_m^\ast(G,A)$ agrees with $\rmH_\m^\ast(K,A)$ above is transparent from the spectral sequence that calculates $\rmH^\ast_\rm{ss}$, but seems to be much harder to prove purely in terms of measurable cochains.

\begin{ques}\label{ques:elem-compar-contractible}
Is there a simple direct proof that $\rmH_\m^p(G,A) = (0)$ whenever $G$ is contractible (equivalently, is a Lie group homeomorphic to a Euclidean space), $A$ is discrete and $p \geq 1$?
\end{ques}

It is not clear how essential is the assumption of almost connectedness in Proposition~\ref{prop:nice-top}.  Even among discrete groups $\G$, for which measurable cochains simply give classical group cohomology, the following natural question seems to be open.

\begin{ques}\label{ques:topologizing-cohom-for-disc-G}
Is there a countable discrete group $\G$ for which $\B^p(\G,\bbZ)$ is not closed in $\Z^p(\G,\bbZ)$ for the topology of pointwise convergence?  Can $\G$ be finitely-generated, or even finitely-presented?  (It is easy to see that this would need $p\geq 2$.)
\end{ques}

More generally, the structure of non-compact totally disconnected groups remains quite mysterious, as do their cohomological properties: see, for instance, Willis~\cite{Wil94} for a thorough discussion.

\begin{prop}\label{prop:almost-ctd-toral}
If $G$ is almost connected and $A$ is toral then each $\B^p(G,A)$ is closed in $\Z^p(G,A)$ and $\rmH_\m^p(G,A)$ is of the form
\[\rm{discrete}\oplus\rm{Euclidean}\]
as a topological group.
\end{prop}

\begin{proof}
First assume that $G$ is connected. Suppose $A = \bbT^k$ and consider the short exact sequence $\bbZ^k\into \bbR^k\onto \bbT^k$. In the resulting long exact sequence we have in part
\begin{multline*}
\cdots\to \rmH_\m^p(G,\bbZ^k) \to \rmH_\m^p(G,\bbR^k) \to \rmH_\m^p(G,A)\\ \stackrel{\rm{switchback}}{\to} \rmH_\m^{p+1}(G,\bbZ^k) \to\cdots.
\end{multline*}

In this case we claim that for $p>0$ the inclusion $\rmH_\m^p(G,\bbZ^k)\to \rmH_\m^p(G,\bbR^k)$ is the zero map. This will imply that the $\bbR^k$-valued cohomology in positive degrees is injected into the $A$-valued cohomology, and Proposition~\ref{prop:top-disc-target} shows that its cokernel is countable. This gives an algebraic isomorphism of the kind we assert, and we can then show that it is also an isomorphism of topological groups.

To see that the above is the zero map, by Proposition~\ref{prop:ThmE-connected} we can replace $\rmH_\m^\ast$ by $\rmH^\ast_{\rm{ss}}$ and look at the spectral sequences for $\rmH^\ast_{\rm{ss}}(G,\bbZ^k)$ and for $\rmH^\ast_{\rm{ss}}(G,\bbR^k)$. The $E_1^{pq}$ terms are $\rmH_\rm{sheaf}^q(G^p,\cal{Z}^k)$ and $\rmH_\rm{sheaf}^q(G^p,\cal{R}^k)$ where $\cal{Z}^k$ and $\cal{R}^k$ are the corresponding sheaves of germs of continuous sections. The sheaf $\cal{R}^k$ is soft so that in the second spectral sequence all terms for $q > 0$ vanish. Therefore in the map from the first spectral sequence to the second, all terms with $q>0$ are mapped to zero. The $q=0$ term is simply the continuous cohomology, but $\rmH_{\rm{cts}}^p(G,\bbZ^k) = 0$ for $p > 0$ because $G$ is connected. This shows that the map 
$\rmH_\m^p(G,\bbZ^k) \to \rmH_\m^p(G,\bbR^k)$ is the zero map for $p>0$.

The proof of topological isomorphism is similar to the preceding propositions. We notice that $\Z^p(G,\bbR^k)$ and $\C^{p-1}(G,A)$ are both continuously mapped into $\Z^p(G,A)$, by the quotient map $q:\Z^p(G,\bbR^k)\to \Z^p(G,A)$ and coboundary operator $d$ respectively. Their sum $W$ is an analytic set as the continuous image of a Borel set, and is a subgroup of countable index, and so just as in the proof of Proposition~\ref{prop:top-disc-target} it must be closed.

By Proposition~\ref{prop:nice-top}, $\B^p(G,\bbR^k)$ is of finite codimension $d$ in $\Z^p(G,\bbR^k)$, and as in the proof of that proposition we can form a Borel cross-section $D$ isomorphic to $\bbR^d$. Then $q(D) + \B^p(G,A) = W$, so $\B^p(G,A)$ is of finite codimension in $W$ and is therefore closed in $W$, again as in the proof of Proposition~\ref{prop:nice-top}. This gives $W/\B^p(G,A) \cong \bbR^d$, and hence that $\rmH_\m^p(G,A)$ is of the form
$(\rm{discrete})\oplus \bbR^d$ in its quotient topology.

Now suppose that $G$ is almost connected and let $G_0$ be its identity component.  By the first part of the proposition we may apply the LHS spectral sequence to obtain
\[E_2^{pq} \cong \rmH_\m^p(G/G_0,\rmH_\m^q(G_0,A))\Longrightarrow \rmH_\m^{p+q}(G,A).\]

We have seen that each $\rmH_\m^q(G_0,A)$ is of the form $D_q\oplus \bbR^{d_q}$ with $D_q$ discrete.  The presentation $\bbR^{d_q}\into \rmH_\m^q(G_0,A)\onto D_q$ is respected by any continuous automorphism of the topological group $\rmH_\m^q(G_0,A)$ because $\bbR^{d_q}$ is mapped to the identity component.  Therefore, by Theorem A, the long exact sequence corresponding to this presentation collapses to show that the quotient maps
\[\rmH_\m^p(G/G_0,\rmH_\m^q(G_0,A))\to \rmH_\m^p(G/G_0,D_q)\]
are all isomorphisms for $p > 0$, and so by Corollary~\ref{cor:disc-cohom-grps} these groups are countable and discrete for $p > 0$.  An easy induction on tableaux now shows that each group $E_i^{pq}$ is still of the form
\[\rm{discrete}\oplus \rm{Euclidean}\oplus \rm{torus},\]
since each nonzero homomorphism that appears in any of the higher tableaux has at least its target group discrete, and so the same is true of $\rmH_\m^{p+q}(G,A)$, as required. \qed
\end{proof}

\begin{cor}
If $G$ is almost connected, $A$ is locally compact and locally connected and the $G$-action on $A$ is trivial then each $\rmH_\m^p(G,A)$ is of the form
\[\rm{discrete} \oplus \rm{Euclidean}\]
as a topological group.
\end{cor}

\begin{proof}
If the $G$-action on $A$ is trivial then, by the Principal Structure Theorem for locally compact Abelian groups (see, for instance, Section 2.4 in Rudin~\cite{Rud62}), any such $A$ decomposes as a $G$-module into a direct sum
\[\rm{discrete}\oplus \rm{Euclidean}\oplus \rm{torus},\]
and so we may simply apply Proposition~\ref{prop:nice-top} to the resulting components of the cocycles. \qed
\end{proof}

The assumption of trivial $G$-action is important in the preceding corollary. Without it, one can construct a locally compact and locally connected module $A$ such that the overall $G$-action does not respect any direct sum structure on $A$, and use this to obtain a non-Hausdorff quotient topology on cohomology.  For example, let $\a,\b \in \bbR$ be irrational and rationally independent, and let $G := \bbR$ act on the group $A := \bbR\times \bbZ^2$ according to
\[T^t(s,m,n) := (s + t\a m + t \b n,m,n).\]
Also let $D := \bbZ\a + \bbZ\b$, so this is a countable dense subgroup of $\bbR$.

\begin{lem}
For this $\bbR$-module $A$ the group $\rmH_\m^1(\bbR,A)$ is topologically isomorphic to the group $\bbR/D$; in particular, it is not Hausdorff.
\end{lem}

\begin{proof}
As usual $\Z^1(G,A)$ is the group of crossed homomorphisms $\bbR\to A$, and this is easily identified with the group of linear maps $\bbR \to A_0 = \bbR$, since any crossed homomorphism $\bbR\to A$ is continuous and hence takes values in $A_0$. Hence $\Z^1(G,A)$ is topologically isomorphic to $\bbR$.  On the other hand, $\B^1(G,A)$ consists of those crossed homomorphisms of the form
\[\a:t\mapsto T^t(s_0,m_0,n_0) - (s_0,m_0,n_0) = t(\a m_0 + \b n_0)\]
for some $(s_0,m_0,n_0) \in A$, and so this is identified with $D$. \qed
\end{proof}

\section{Completion of the comparison theorem}\label{sec:compar-complete}

\noindent\emph{Proof of Theorem E}\quad As remarked previously, in case $A$ is Fr\'echet this follows at once from Theorem A and Wigner's Theorem 3 in~\cite{Wig73}, so we focus on the case of locally compact and locally contractible $A$.  Such a $A$ lies in Segal's category of modules, giving directly that $\rmH^\ast_\rm{ss}(G,A)\cong \rmH^\ast_\rm{Seg}(G,A)$.

Suppose first that $A$ is discrete. Wigner's Theorem 4 in~\cite{Wig73} includes the construction of $\psi:\rmH^\ast_\rm{ss}(G,A)\to \rmH^\ast_\rm{cs}(G,A)$ and the proof that it is an isomorphism; he imposes the condition that $G$ be finite-dimensional, but in fact he uses this only for his earlier proof of the agreement with $\rmH_\m^\ast(G,A)$. This isomorphism can also be obtained by composing the agreement of $\rmH^\ast_\rm{ss}$ and $\rmH^\ast_\rm{Seg}$ discussed in the Introduction with Segal's isomorphism $\rmH^\ast_\rm{Seg}\to \rmH^\ast_\rm{cs}$ from~\cite{Seg70}.

For $\a:\rmH^\ast_\rm{ss}(G,A)\to \rmH^\ast_\m(G,A)$, Proposition~\ref{prop:ThmE-connected} already shows that it is an isomorphism if $G$ is connected.  In general let $G_0 \unlhd G$ be the identity component.  By Proposition~\ref{prop:top-disc-target} the groups $\rmH_\m^\ast(G_0,A)$ are discrete in their quotient topologies, and hence the LHS spectral sequence of~\cite{Moo76(gr-cohomIV)} is available to give the abutment
\[E^{pq}_2 \cong \rmH_\m^p(G/G_0,\rmH_\m^q(G_0,A))\Longrightarrow \rmH_\m^{p+q}(G,A).\]

On the other hand, since $G/G_0$ is totally disconnected, by Michael's Theorem the inclusion $G_0 \leq G$ has a local (in fact global) cross-section, and so the theory $\rmH^\ast_\rm{ss}$ has an analogous spectral sequence.  Since 
\[\rmH_\m^\ast(G_0,A) \cong \rmH^\ast_\rm{ss}(G_0,A)\]
by Proposition~\ref{prop:ThmE-connected}, and
\[\rmH_\m^\ast(G/G_0,B) \cong \rmH^\ast_\rm{cts}(G/G_0,B) \cong \rmH^\ast_\rm{ss}(G/G_0,B)\]
for any discrete module $B$ by the comparison result for totally disconnected acting groups (Theorem 1 in~\cite{Wig73}), these two spectral sequences show that $\a$ is an isomorphism.

Finally, consider an arbitrary locally compact and locally contractible $A$.  By the structure theorem, $A$ must be of the form
\[\rm{discrete} \oplus \rm{compact}\oplus \rm{Euclidean},\]
and this can be locally contractible only if the compact part is a finite extension of a torus. Letting $A_0$ be the identity component of $A$, this implies that $A_0$ is a connected Abelian Lie group and that $A/A_0$ is discrete.

It therefore suffices to prove that $\a$ is an isomorphism for the three cases of discrete, toral and Euclidean $A$ separately, since then we can patch them together via the long exact sequence of cohomology using the Five Lemma.  Moreover, since tori are quotients of Euclidean spaces by discrete subgroups (and their automorphisms all lift to the covering Euclidean spaces), it actually suffices to treat the first two cases.  Euclidean $A$ are handled by Theorem A and discrete $A$ by the argument above, so the proof is complete. \qed

\vspace{7pt}

In view of these positive results, it seems worth including an example of groups $G$ and $A$ for which $\rmH_\m^\ast$ and $\rmH^\ast_\rm{ss}$ do not agree.  Let $G :=\bbT^\infty$ and $A := (\bbZ/2\bbZ)^\infty$ endowed with the trivial $G$-action.  In this case the groups $\rmH^\ast_\rm{ss}(G,A)$ still enjoy some continuity under the inverse limit of the quotients $\bbT^\infty \to\bbT^N$; this follows from Proposition 8.1 of Flach~\cite{Fla08} (one also needs his Corollary 7 to verify the hypotheses needed by that proposition, and that corollary is stated only for discrete modules, but the proof is the same for totally disconnected modules).

However, $\rmH_\m^\ast$ is \emph{not} continuous in this instance. Since $\rmH_\m^2(G,A)$ classifies compact extensions of $A$ by $G$, we may consider the class corresponding to the extension
\[(\bbZ/2\bbZ)^\infty \into \bbT^\infty \stackrel{\times 2}{\onto} \bbT^\infty.\]
This class cannot be in the image of the inflation map arising from any quotient $\bbT^\infty \to\bbT^N$, since then the corresponding extension would have a quotient homomorphism onto some compact extension of $(\bbZ/2\bbZ)^\infty$ by $\bbT^N$, and it is easily seen that such a group would have to be locally disconnected, whereas this is not true of $\bbT^\infty$.

(Thus, in this example $\rmH^\ast_\rm{ss}$ appears to behave `better' than $\rmH_\m^\ast$, but it is because only $\rmH_\m^\ast$ is able to capture the complicated topological structure of this group extension. A key feature of that structure is that it has no continuous local cross-section.)

Recall that Wigner in~\cite{Wig73} proved that $\rmH^\ast_\m(G,A) \cong \rmH^\ast_\rm{ss}(G,A)$ provided $A$ has his property F and $G$ is finite-dimensional. In this case $A$, being compact, has property F, but $G$ is clearly not finite-dimensional.

\vspace{7pt}

\noindent\emph{Remark}\quad J.L. Tu's recent paper~\cite{Tu06} contains the assertion that $\rmH_\m^\ast$ and $\rmH^\ast_\rm{ss}$ agree for all locally compact $G$ and Polish $G$-modules $A$.  In view of the above example this is incorrect, and it seems that there are some problems with Section 6 of his paper (specifically, we suspect, with his Proposition 6.1(b)).  Nevertheless, his construction of a sheaf-based cohomology for topological groupoids has other interesting consequences, and may contribute to other comparison results in the future: see the last part of Section~\ref{sec:further-discuss} below. \fin

\section{Further discussion}\label{sec:further-discuss}

The proofs of Theorems A--E suggest various other avenues for exploration, and points at which an alternative argument might be desirable.  In this section we offer some remarks on a few of these.

\subsection{More explicit comparisons with classifying space cohomology}

The known instances of isomorphism $\rmH_\m^\ast(G,A)\cong \rmH^\ast_\rm{cs}(G,A)$ for discrete $A$ and various $G$ are all obtained via $\rmH^\ast_\rm{ss}$ (and possibly also $\rmH^\ast_\rm{Seg}$, according as one chooses to follow Wigner's or Segal's construction of the maps landing in $\rmH^\ast_\rm{cs}$).  It seems instructive (and may be computationally useful) to include a more explicit construction of a comparison map
\[\rmH^\ast_\rm{cs}(G,A) \to \rmH_\m^\ast(G,A),\]
although we will not offer a correspondingly explicit proof that it is an isomorphism.  This will be the map $\iota_2$ appearing in our earlier diagram among the theories.

Recall that having fixed a choice of classifying bundle $E_G \stackrel{\pi}{\to} B_G$ for $G$, to a Polish $G$-module $A$ we associate the sheaf $\cal{A}$ on $B_G$ defined by
\[\A(U) := \rm{Map}(\pi^{-1}(U),A)^G,\]
and we have so far interpreted $\rmH^\ast_\rm{cs}(G,A)$ as $\rmH^\ast_{\rm{sheaf}}(G,\A)$. This may depend on the choice of $B_G$ in case $A$ is not discrete, so the notation is a little ambiguous.  For the comparison maps constructed by Wigner or Segal the case of non-discrete $A$ is needed during the dimension-shifting induction, so we still need to introduce it.

However, let us now change tack slightly.  We will suppose that $A$ is any $G$-module, but that the $G$-action is trivial, and next consider a much more elementary comparison map from an alternative $A$-valued cohomology on $B_G$. This new cohomology theory on $B_G$ is rather non-standard but is well-adapted to our purpose.

First observe that for locally compact, second countable groups $G$, the standard constructions of $E_G$ (such as Milnor's, which can be found in~\cite{Hus95}) all give it the structure of an increasing union of locally compact metrizable spaces $E_G^{(n)}$ with the direct limit topology, where $E_G^{(n)}$ is already $n$-connected.  Let $B_G^{(n)}$ be the corresponding quotients, so the topologies on these may all be generated by Polish metrics. 

For any Polish space $X$, let $S_k(X)$ denote the collection of all singular $k$-simplices in $X$: that is, continuous maps $\Delta^k\to X$ from the standard $k$-simplex $\Delta^k \subset \bbR^{k+1}$.  Then $S_k(X)$ is also Polish in its uniform topology, and this gives rise to a standard Borel structure on $S_k(X)$.

The union $E_G = \bigcup_{n\geq 1}E_G^{(n)}$ is generally not Polish, but if we let
\[S_k(E_G) := \bigcup_{n\geq 1}S_k(E_G^{(n)})\]
and similarly for $S_k(B_G)$ then each of these still carries a natural standard Borel structure obtained from the $E_G^{(n)}$.  Note that in general $S_k(E_G)$ is strictly smaller than the set of all singular $k$-simplices in the union $E_G$, since there may be such simplices that touch $E_G\setminus E_G^{(n)}$ for infinitely many $n$.

Now, if the $G$-action on $A$ is trivial, then classical singular cohomology $\rmH^\ast_{\rm{sing}}(B_G,A)$ is defined using arbitrary maps from $S_k(B_G)$ to the group $A$. By a simple analogy with this, we may compute a new cohomology theory called $\rmH^\ast_{\rm{m.sing}}(B_G,A)$ in which we insist that all singular cochains be measurable.  This is well-defined by the easy observation that the boundary of a measurable cochain is still measurable.

In general it is not clear when $\rmH^\ast_{\rm{m.sing}}$ agrees with singular or \v{C}ech cohomology.  Certainly the usual isomorphism results also cover $\rmH^\ast_{\rm{m.sing}}$ on the category of CW-complexes, but among arbitrary locally compact spaces we suspect that all three theories can differ.  However, for discrete $A$ it turns out to be easiest to compare $\rmH^\ast_\rm{m.sing}(B_G,A)$ with $\rmH_\m^\ast(G,A)$; the universality properties of $\rmH^\ast_\rm{sheaf}(B_G,\A)$ then give a comparison map from that to $\rmH^\ast_\rm{m.sing}(B_G,A)$ if desired.  At least for Lie groups one knows that each $B_G^{(n)}$ (in any of the standard constructions) is a finite-dimensional CW-complex, and so that is a case in which all the spatial cohomology theories agree.

Hence our goal is to construct a map
\[\iota_2:\rmH^\ast_\rm{m.sing}(B_G,A)\to \rmH^\ast_\m(G,A).\]
We construct this via the auxiliary cochain complex: \[\big(\rm{Map}_\rm{m}(S_k(E_G),A)^G\big)_{k\geq 1}\]
of $G$-invariant measurable maps from $S_k(E_G)$ to $A$.  The point is that the homology of this cochain complex can be shown always to coincide with $\rmH_\m^\ast(G,A)$ using the usual construction of a chain homotopy relating group and classifying space cohomology, with some judicious appeals to the Measurable Selector Theorem.

First, pick a distinguished fibre $i(G)$ in $E_G$, and let
\[a_0 = i: G \to E_G\quad\hbox{and}\quad b_0: E_G = S_0(E_G) \to G\]
be measurable and $G$-equivariant maps; the choice of $b_0$ is equivalent to choosing a measurable section of $\pi$, which again can be done explicitly using a suitable model of $E_G$.  Now one defines by induction on $k$ two families of $G$-equivariant measurable maps
\[a_k: G^{k+1} \to S_k(E_G)\quad\hbox{and}\quad b_k: S_k(E_G) \to G^{k+1}\]
so that $a_k\circ\partial = \partial\circ a_{k+1}$ and similarly for $b_k$.

The $b_k$ are easy to construct once we have $b_0$: given a $k$-simplex $f:\Delta^k \to E_G^{(n)}$, one simply sends it to the $(k+1)$-tuple of $b_0$-images of its vertices.

The $a_k$ are constructed using the fact that each $E_G^{(n)}$ is $n$-connected, so that on the set of those finite tuples in $S_k(E_G^{(n)})$ that define singular cycles (that is, the formal sums of their boundaries are zero) one can make a measurable choice of finite tuples in $S_{k+1}(E_G^{(n)})$ having those
cycles as their boundaries, provided $n$ is larger than $k$.  Since the $G$-action on each $S_k(E_G^{(n)})$ is clearly smooth (in Mackey's sense) this selection can be made $G$-equivariantly.  Given this, and if we already know $a_k$, then $a_{k+1}$ of some $(k+1)$-tuple $(g_0,...,g_{k+1})$ is simply a measurable selection
of a singular $(k+1)$-simplex in $E_G^{(n)}$ for some large $n$ that has boundary equal to the formal sum of
the $a_k$-images of the boundary of $(g_0,...,g_{k+1})$.

Having made these definitions, one considers the differences
\[\rm{id} - a_k \circ b_k\quad \hbox{and}\quad   b_k \circ a_k - \rm{id}\]
and can run through a standard proof that these are both chain
homotopic to the identity maps on their respective complexes.
(Ultimately one wants chain homotopies on the cochain complexes
instead, but this is obtained simply by composing cochains
with these measurable maps on the spaces of $k$-tuples and $k$-simplices.)  We omit a full description of this here, but it is a classical construction from the foundations of homology theory, with the only added twist that selections must be made measurably: a quick check shows that this is always possible.  A careful presentation of this classical argument appears as Theorem 2.10 in
Hatcher's book~\cite{Hat02}.

Thus we obtain an isomorphism
\[\rmH_\m^\ast(G,A) \cong H^\ast\big(\rm{Map}_\rm{m}(S_\ast(E_G),A)^G\big).\]
Now the comparison map from $\rmH^\ast_{\rm{m.sing}}(B_G,A)$ is obvious: given any measurable cocycle $f:S_k(B_G)\to A$, simply lifting this through the projection $\pi_\ast:S_k(E_G)\to S_k(B_G)$ gives a $G$-invariant measurable cocycle $S_k(E_G)\to A$.  This lifting defines a sequence of homomorphisms
\[\rm{Map}_\rm{m}(S_k(B_G),A)\to \rm{Map}_\rm{m}(S_k(E_G),A)^G,\quad k\geq 1,\]
which commute with coboundaries, and hence a chain map which descends to maps on cohomology
\[\iota_2:\rmH^\ast_{\rm{m.sing}}(B_G,A)\to H^\ast\big((\rm{Map}_\rm{m}(S_\ast(E_G),A)^G\big).\]

In the case of discrete groups $G$, and provided $\rmH^\ast_\rm{m.sing}$ agrees with some more conventional spatial cohomology, this is precisely the usual isomorphism between $\rmH^\ast(B_G,A)$ and $\rmH_\m^\ast(G,A)$.  The reason it gives an isomorphism there is simple and classical:
\begin{quote}
If $G$ is discrete, then for every $k\geq 1$, any element of $S_k(B_G)$ has a unique lift to an element of $S_k(E_G)$ up to translation by the $G$-action on $E_G$.
\end{quote}
As a result, the lifting by $\pi$ actually defines an isomorphism
\[\rm{Map}_\rm{m}(S_k(B_G),A)\cong \rm{Map}_\rm{m}(S_k(E_G),A)^G\]
for every $k$.

In fact, the same holds for totally disconnected $G$, and so the same argument also gives the following.

\begin{prop}
If $G$ is totally disconnected and $A$ is any Polish $G$-module with trivial action then
\[\rmH_\m^\ast(G,A) \cong \rmH^\ast_\rm{m.sing}(B_G,A).\]
\qed
\end{prop}

\noindent\emph{Remark}\quad By working with local-coefficient cohomology (see, for instance, Section 3.H of Hatcher~\cite{Hat02}), it should be possible to extend the above definition of $\rmH^\ast_\rm{m.sing}(B_G,A)$ to the case of a $G$-module $A$ on which the $G$-action factorizes through the totally disconnected quotient $G/G_0$.  To go beyond this we believe one is forced to return to the more flexible setting of sheaf cohomology, since the classical definition of cohomology with local coefficients makes essential use of a uniqueness among lifts of paths into the acting group, which may fail in case $G_0$ acts on $A$ nontrivially. \fin

\vspace{7pt}

In general, if $G$ has nontrivial path components then singular $k$-simplices in $B_G$ do not have unique lifts in $S_k(E_G)$ modulo the $G$-action for $k\geq 1$.  As a result, for non-totally-disconnected $G$ and $k\geq 1$ we can imagine the following two problems:
\begin{itemize}
\item On the one hand, there can be several $G$-orbits in $S_k(E_G)$ which project onto the same element of $S_k(B_G)$, but there is no reason why measurable cocycles in $\rm{Map}_\rm{m}(S_k(E_G),A)^G$ should take the same values on those different $G$-orbits.  There may therefore be nontrivial classes in $\rmH_\m^\ast(G,A)$ none of whose representatives can arise by lifting from $B_G$.
\item On the other hand, there may be measurable cochains in some $\rm{Map}_\m(S_k(E_G),A)^G$ which can be lifted from $\rm{Map}_\rm{m}(S_k(B_G),A)$, but which are coboundaries upstairs of cochains which cannot be so lifted, so that their cohomology classes are killed by this lift.
\end{itemize}

In light of these potential difficulties, the full strength of Theorem E seems very surprising.  The proofs we have of that theorem rely on comparisons via $\rmH^\ast_\rm{ss}$, but it would surely be enlightening to see a more elementary proof.

\begin{ques}\label{ques:elem-compar-for-discrete}
Is there a more elementary proof that for any locally compact, second countable $G$ and a discrete $G$-module $A$ with trivial action one has
\[\rmH_\m^\ast(G,A) \cong \rmH^\ast_\rm{m.sing}(B_G,A) \cong \rmH^\ast_\rm{sheaf}(B_G,\cal{A})?\]
\end{ques}

\subsection{Two illustrative examples}

We offer two examples to fill out the above discussion.

In the first, $A$ is discrete but $\rmH^\ast_{\rm{m.sing}}(B_G,A)$ does not agree with $\rmH^\ast_\rm{sing}(B_G,A)$: the former correctly computes $\rmH_\m^\ast(G,A)$ but the latter does not.

Consider the compact group $G = (\bbZ/2\bbZ)^\infty$, and let $G_\d$ denote the same group with its discrete topology.  If we use Milnor's construction of classifying spaces then it is functorial in the groups: in particular, the continuous isomorphism (with non-continuous inverse) given by the identity $G_\d \to G$ gives a continuous bijection (with non-continuous inverse) $E_{G_\d}\to E_G$.  However, it is easy to check from the total disconnectedness of $G$ that given a basepoint $x \in E_G$, a preimage of it $y \in E_{G_\d}$, and a singular simplex $f:\Delta^k\to E_G$ based at $x$, there is a unique lifted simplex in $E_{G_\d}$ based at $y$.  Hence there is a canonical bijection between the simplicial complexes $S_\ast(E_{G_\d})$ and $S_\ast(E_G)$ (but not between their own topologies or Borel structures).  By classical discrete-group cohomology theory, the $(\!\!\!\!\mod 2)$ singular cohomology of $B_{G_\d}$ computes $\rmH^\ast(G_\d,\bbZ/2\bbZ)$, and so in view of this bijection between simplicial complexes the same is true of the $(\!\!\!\!\mod 2)$-singular cohomology of $B_G$.

On the other had, since this $G$ is finite- (indeed, zero-) dimensional we have seen that the $(\!\!\!\!\mod 2)$ measurable singular cohomology of $B_G$ computes $\rmH_\m^\ast(G,\bbZ/2\bbZ)$.  This differs from $\rmH^\ast(G_\d,\bbZ/2\bbZ)$: for instance, in degree $1$, the latter contains all linear functionals $(\bbZ/2\bbZ)^\infty\to \bbZ/2\bbZ$, whereas the former contains only the continuous such functionals.

Of course, Theorem E applies to this example, so in this case $\rmH^\ast_\rm{m.sing}(B_G,A)$ does agree with $\rmH^\ast_{\rm{sheaf}}(B_G,\A)$.

Our second example pertains to the intuitive objection to using $\rmH^\ast(B_G,A)$ for non-discrete $A$ that it does not correctly take the topology of $A$ into account. In this example, the map
\[\rmH^\ast_\rm{m.sing}(B_G,\bbR) \to H^\ast\big(\rm{Map}_\rm{m}(S_\ast(E_G),\bbR)^G\big)\]
will fail to be an isomorphism, showing that this problem is real even if one insists on using measurable singular cochains in the above development.

The example is offered by $\rmH_\m^2(\bbT,\bbR)$. By Theorem A this is $(0)$, but on the other hand standard calculations give $\rmH^2(B_\bbT,\bbR) \cong \bbR$. Indeed, the spatial cohomology appearing here may be taken as singular or \v{C}ech, since a model of $B_\bbT$ up to homotopy is the infinite-dimensional complex projective space $\bbC\rm{P}^\infty$ and this is a direct limit of finite CW-complexes.  The total space of the classifying principal $\bbT$-bundle is given by the union of the complex spheres $S^{2n-1} \subset \bbC^n$ with the diagonal multiplication action of $\bbT \cong \rm{U}(1)$.  Since $\bbC\rm{P}^\infty$ is a $\rm{K}(\bbZ,2)$, its real cohomology ring is $\bbR[X^2]$ (that is, a real polynomial algebra generated by a single formal variable in degree $2$), so is $\bbR$ in degree $2$.

Let us sketch a more concrete proof of this calculation.  We will omit technical details, since this example is not central to our work, and assume some familiary with differential topology. First, one has
\[B_\bbT = \bbC\rm{P}^\infty = \bigcup_{m\geq 1}\bbC\rm{P}^m\]
under the obvious inclusions, and the $k^\rm{th}$ cohomology group of this increasing union stabilizes once $2m + 1 \geq k$.  (The basic facts about $\bbC\rm{P}^\infty$, including its r\^ole as $B_\bbT$ and this last calculation, can be found in Section 6.14 of Davis and Kirk~\cite{DavKir01}.)  It therefore suffices to compare $\rmH^2(\bbC\rm{P}^1,\bbR)$ with $\rmH_\m^2(\bbT,\bbR)$.  Arguing similarly, the homology of the cochain complex $\rm{Map}_\m(S_\ast(E_\bbT),\bbR)^\bbT$ in degree $2$ may already be computed using the subset of unit vectors $S^3 \subset \bbC^2$ with the action of $\bbT\cong \rm{U}(1)\subset \bbC$ by multiplication, so that
\[\rmH_\m^2(\bbT,\bbR)\cong \frac{\rm{ker}\big(d:\rm{Map}_\rm{m}(S_2(S^3),\bbR)^{\rm{U}(1)} \to \rm{Map}_\rm{m}(S_3(S^3),\bbR)^{\rm{U}(1)}\big)}{\rm{im}\big(d:\rm{Map}_\rm{m}(S_1(S^3),\bbR)^{\rm{U}(1)} \to \rm{Map}_\rm{m}(S_2(S^3),\bbR)^{\rm{U}(1)}\big)}.\]

Thus we have reduced the comparison of the two theories to calculations concerning the $\rm{U}(1)$-principal bundle
\[\rm{U}(1) \into S^3\stackrel{\pi}{\to} \bbC\rm{P}^1 \cong S^2.\]
This is none other than the second Hopf fibration, and it turns out that the relevant calculations here are part of a classic example in differential topology: the computation of the associated Hopf invariant.  This may be found, for instance, as Example 17.23 in Bott and Tu~\cite{BotTu82}; here we only record the results.  Since these spaces are all compact smooth manifolds, we may compute $\rmH^2(\bbC\rm{P}^1,\bbR)$ using de Rham cohomology, and a routine adaptation of that comparison to de Rham theory shows that in the same way one has
\begin{eqnarray}\label{eq:H2TR}
\rmH_\m^2(\bbT,\bbR) \cong \frac{\rm{ker}\big(d:\O^2(S^3)^{\rm{U}(1)}\to \O^3(S^3)^{\rm{U}(1)}\big)}{\rm{im}\big(d:\O^1(S^3)^{\rm{U}(1)} \to \O^2(S^3)^{\rm{U}(1)} \big)}.
\end{eqnarray}

As is standard, $\rmH^2_\rm{dR}(S^2)\cong \bbR$ with a canonical generator given by the surface-area form $\omega$.  Re-writing this as a form on $\bbC\rm{P}^1$ using stereographic projection and then lifting it through $\pi$ gives a de Rham $2$-cocycle on $S^3$. In the coordinates
\[S^3 = \{(x_1 + \rm{i}x_2,x_3 + \rm{i}x_4)\in \bbC^2:\ x_1^2 + x_2^2 + x_3^2 + x_4^2 = 1\},\]
this $2$-cocycle takes the form
\[\pi^\ast\omega = \frac{1}{\pi}(dx_1 \wedge dx_2 + dx_3 \wedge dx_4),\]
and so it equals the coboundary of the $1$-form
\[\a := \frac{1}{2\pi}(x_1 \wedge dx_2 - x_2 \wedge dx_1 + x_3 \wedge dx_4 - x_4 \wedge dx_3).\]
Now a simple calculation shows that this is invariant under the $\rm{U}(1)$-action, but it cannot have been lifted from a $1$-form on $S^2$ because $\omega$ is not a coboundary there and $\pi^\ast$ is injective on $\O^2(S^2)$.

The above discrepancy can be understood in another way by observing that $\rmH^\ast(B_\bbT,\bbR)$ is really computing $\rmH_\m^\ast(\bbT,\bbR_\d)$, the measurable-cochains cohomology of $\bbT$ with values in $\bbR$ with its \emph{discrete} topology.  This is because the class in
\[H^2\big(\rm{Map}_\rm{m}(S_\ast(E_\bbT),\bbR)^\bbT\big)\]
that corresponds to the $\rm{U}(1)$-invariant de Rham $2$-cocycle $\pi^\ast\omega$ has a representative which takes only finitely many values on $S_2(E_\bbT)$ (indeed, it arises by tensorizing the space of $\bbZ$-valued cochains with $\bbR$), and so it is measurable as a map $S_2(E_\bbT) \to \bbR_\d$.  However, the cochain whose boundary is this class in $\rm{Map}_\rm{m}(S_2(E_\bbT),\bbR)^\bbT$ can\emph{not} be chosen to take only a discrete set of values, and so it is not measurable as a map $S_1(E_\bbT) \to \bbR_\d$.  Thus for $\bbR_\d$-valued cohomology, this $2$-cochain is not a coboundary, and we obtain the correct calculation
\[\rmH^2_{\rm{dR}}(B_\bbT,\bbR)\cong \bbR \cong \rmH_\m^2(\bbT,\bbR_\d)\]
(taking the liberty of writing $\rmH_\m^2(\bbT,\bbR_\d)$ for the Borel-cochains theory, even though $\bbR_\d$ is not separable and hence not Polish).

So the difference between $\rmH^2(B_\bbT,\bbR)$ and $H^2(\rm{Map}_\m(S_\ast(E_\bbT),\bbR)^\bbT)$ is reflecting the difference between $\rmH_\m^2(\bbT,\bbR)$ and $\rmH^2_\m(\bbT,\bbR_\d)$.  To complete the picture, that difference can be described in terms of their group-extension interpretation.  Any extension
\[(0) \to \bbR_\d \stackrel{i}{\to} E \to \bbT \to (0)\]
is a Lie group because the map from $E$ to $\bbT$ is a local homeomorphism.  It is one-dimensional and has continuum-many connected components. So $E$ has a unique one-parameter subgroup with parameter lifted from $\bbT$, which looks like a helix.  If one follows this one-parameter subgroup as its image moves around $\bbT$ and returns to the identity, upstairs one returns to the subgroup $i(\bbR_\d)$ at some element $a\in \bbR_\d$. This element is an invariant of the extension, and in fact it parameterizes the possible extensions. So $\rmH^2_\m(\bbT, \bbR_\d) = \bbR_\d$.  The resulting $2$-cocycle takes only the values $0$, $a$ and $-a$. If we now map $\bbR_\d$ continuously to $\bbR$, then $E$ is mapped to an extension of $\bbR$ by $\bbT$.  Its image is a $2$-dimensional Lie group, and if we multiply the one-parameter subgroup above by the one-parameter subgroup in $\bbR$ whose time-$1$ image is $-a$, then their product is a one-parameter subgroup that closes up at time $1$ and trivializes the extension. This corresponds to finding a $1$-cochain whose coboundary is the image in $\rmH^2_\m(\bbT,\bbR)$ of the $2$-cocyle above.  That $1$-cochain is a segment of a local homomorphism of $\bbT$ into $\bbR$, so it takes an uncountable number of values and is not a Borel map of $\bbT$ into $\bbR_\d$.

Calculations with differential forms of the kind sketched above are a general feature of equivariant de Rham theory for actions of compact Lie groups on manifolds.  This relates the de Rham cohomology of the Borel construction $M_G = (E_G\times M)/G$ for a $G$-manifold $M$ (which is one way of defining the equivariant cohomology of this action) to the cohomology of a complex of differential forms on $E_G\times M$.  However, this link is made by considering those forms on $E_G\times M$ that are both invariant under the $G$-action and also are annihilated by interior multiplication with a certain family of vector fields constructed from the $G$-action. The complex of forms on $E_G\times M$ that are only assumed to be $G$-invariant is still too large, and its cohomology generally does not agree with $\rmH^\ast_\rm{dR}(M_G)$, in much the same way that the $\rm{U}(1)$-invariant forms that we used to express $\rmH_\m^2(\bbT,\bbR)$ in equation~(\ref{eq:H2TR}) give a different calculation from the de Rham cohomology $\rmH^\ast_\rm{dR}(S^2)$ of the quotient space.  Thus, and perhaps surprisingly, the restriction to $\rm{U}(1)$-invariant forms on $S^3$ that represent classes in $\rmH_\m^2(\bbT,\bbR_\d)$, which do give a calculation agreeing with $\rmH^2_\rm{dR}(S^2)$, is equivalent to the restriction to those $\rm{U}(1)$-invariant differential forms that satisfy an additional condition of annihilation by interior multiplication with a certain list of vector fields.  The monograph~\cite{GuiSte99} gives a thorough development of the de Rham side of this general theory.

\subsection{Another possible approach to the classifying space comparison}

Wigner's  proof of the isomorphism $\rmH_\rm{ss}^\ast(G,A) \cong \rmH^\ast(B_G,\cal{A})$ for discrete $G$-modules $A$ is a little difficult to grasp conceptually, but we note that the cohomology of groupoids as developed by J.L. Tu in~\cite{Tu06} may shed some light conceptually on what is going here.  Tu's cohomology groups are a direct generalization of the theory $\rmH^\ast_\rm{ss}$ for groups.

Suppose again that $G$ is a locally compact, second countable group.  Then it acts freely and properly on the space $E_G$ and so we can form the groupoid corresponding to this action, $G \times E_G$. Because the action is free, this is a principal groupoid --- that is, it is an equivalence relation on its unit space $E_G$, and it is the equivalence relation  generated by the action of $G$. This is not a locally compact groupoid so Tu's theory does not apply, although it does have a Haar system, which Tu also requires. It would seem that Tu's theory could be extended to cover groupoids like this, but here we shall simply assume that without exploring it in detail. In any case there is homomorphism of groupoids $p:G \times E_G \to G$ defined by $p(g,e) := g$, and the discrete $G$-module $A$ defines a discrete $(G \times E_G)$-module $A^E$. Then there is a homomorphism (an inflation homomorphism in cohomology)
\[p^\ast: \rmH_\rm{ss}^\ast(G,A) \to \rmH_\rm{ss}^\ast(G \times E_G,A^E),\]
where $\rmH_\rm{ss}^\ast$ also denotes what would be Tu's groupoid cohomology groups here.

There is also a groupoid homomorphims $\pi: G \times E_G \to B_G$ where $B_G$ is a principal groupoid corresponding to the trivial equivalence relation on $B_G$ defined by the diagonal. The $G$-module $A$ defines a locally constant sheaf $\cal{A}$ on $B_G$ which is a $B_G$-module. Then there is a pullback map
\[\pi^\ast: \rmH_\rm{ss}^\ast(B_G,\cal{A}) \to \rmH_\rm{ss}^\ast(G \times E_G, A^E).\]
But $\rmH_\rm{ss}^\ast(B_G,\cal{A})$ is just the ordinary sheaf cohomology of the space $B_G$ (c.f. Proposition 4.8 in~\cite{Tu06}).

One would then argue that the first map $p^\ast$ is an isomorphism because $E_G$ is contractible and $A^E$ is a constant sheaf, and then proceed by comparing the spectral sequences for these two cohomology groups. This is exactly what is happening in the first part of Wigner's proof. Then the second map $\pi^\ast$ should also be an isomorphism because $\pi$ is a Morita equivalence (c.f. Proposition 8.1 in Tu~\cite{Tu06}). Thus $\rmH_\rm{ss}^\ast(G,A)$ is isomorphic to $\rmH^\ast(B_G,\cal{A})$ because both are isomorphic to $\rmH_\rm{ss}^\ast(G \times E_G, A)$.

This is not a proof because, as remarked above, Tu's theory is not known to apply here, but it does shed some conceptual light on the result.  The hypothesis that $A$ is discrete is essential to assure that the sheaves on $G$, $E_G$ and $B_G$  corresponding to $A$ are constant or locally constant sheaves and so all line up correctly, and that the sheaf cohomology that enters is homotopy-invariant.

\subsection{Further consequences for representatives of cohomology classes}

Suppose now that $G$ is compact and $A$ discrete. By combining Theorems B and C, Proposition~\ref{prop:cpt-and-disc} and the explicit description of the comparison map $\rmH^\ast(B_G,A) \to \rmH_\m^\ast(G,A)$ obtained above, we can prove that all cohomology classes in $\rmH_\m^\ast(G,A)$ must have representatives of a very restricted kind.

\begin{dfn}
If $G$ is a compact group and $A$ a discrete group then a function $\s:G^p\to A$ is \textbf{semi-algebraic} if it takes only finitely many values, and if there is a finite-dimensional representation $\rho:G\to \rm{U}(n)$ such that $\s = \tau\circ \rho^{\times p}$ for a function $\tau:\rm{U}(n)^p\to A$ whose level sets are semi-algebraic (in the sense of Real Algebraic Geometry; see, for instance,~\cite{BocCosRoy98}).
\end{dfn}

\begin{prop}\label{prop:semi-alg-rep}
If $G$ is compact, $A$ is discrete and $p \geq 1$ then every class in $\rmH_\m^p(G,A)$ has a semi-algebraic representative.
\end{prop}

Of course this representative need not be unique.

\begin{proof}
By Theorem B, any class in $\rmH_\m^p(G,A)$ is inflated from some quotient $\pi:G\to G_1$ with $G_1$ a Lie group, which may be embedded as subgroup of some $\rm{U}(n)$. We may therefore assume that $G$ is itself a Lie subgroup of $\rm{U}(n)$, and furthermore that it is a closed real algebraic subgroup (using the Structure Theory for compact Lie groups).

Now, by Proposition~\ref{prop:cpt-and-disc} the comparison map $\rmH^p(B_G,A)\to \rmH_\m^p(G,A)$ constructed above is an isomorphism.  Moreover, since discrete-valued cohomology groups are homotopy-invariant, we may pick any suitable model for $B_G$ in making this comparison.  In particular, $G$ acts freely on the total space $E_{\rm{U}(n)}$ of the universal bundle for the unitary group that contains it, and so we may take $E_{\rm{U}(n)}/G$ for our model of $B_G$.

However, an explicit choice for $E_{\rm{U}(n)}$ is available in the form of the infinite-dimensional complex Stiefel manifold
\[V_n(\bbC^\infty) := \bigcup_{m\geq n}V_n(\bbC^m),\quad\quad  V_n(\bbC^m) := \rm{U}(m)/\rm{U}(m-n),\]
(see, for instance, Subsection 6.14.3 in Davis and Kirk~\cite{DavKir01}). Moreover, in order to compute degree-$p$ cohomology it suffices to consider the smaller $G$-bundle
\[G\into V_n(\bbC^m) \to V_n(\bbC^m)/G\]
for $m > p + n + 1$, since in this case $V_n(\bbC^m)$ is already $p$-connected and its cohomology groups up to degree $p$ have stabilized.  Since $G$ is a closed real algebraic subgroup of $\rm{U}(n)$, and so $G\times \rm{U}(m-n)$ is a closed real algebraic subgroup of $\rm{U}(m)$, the quotient
\[V_n(\bbC^m)/G = \rm{U}(m)/(G\times \rm{U}(m-n)) \subset M_{m\times m}(\bbC)/(G\times \rm{U}(m-n))\]
still carries the structure of a real algebraic manifold (see, for instance, Theorem 3.4.3 in Onishchik and Vinberg~\cite{OniVin90}).

Hence we obtain an isomorphism
\[\rmH^p(V_n(\bbC^m)/G,A) \to \rmH_\m^p(G,A),\]
where to be explicit we choose a single $G$-orbit $\iota:G\into V_n(\bbC^m)$ and then map each $(p+1)$-tuple in $G$ to a suitable singular $(p+1)$-simplex in $V_n(\bbC^m)$ with vertices equal to the images of that $(p+1)$-tuple under $\iota$.  However, $V_n(\bbC^m)/G$ is a compact real algebraic manifold, and therefore its cohomology may also be computed as the simplicial cohomology of any suitable triangulation.  This, in turn, may be chosen so that its cells are semi-algebraic, and having done this the resulting map from simplicial $p$-cochains on this triangulation of $V_n(\bbC^m)/G$ to measurable cochains on $G^p$ results in functions that are semi-algebraic, as required. \qed
\end{proof}

It seems likely that another dimension-shifting argument (similar to that underlying Proposition~\ref{prop:basic-smoothness}) could give a more direct proof of this assertion, and some related results on the structure of cocycles for more general groups $G$.   We will not explore it further in detail, but to illustrate its other applications we offer a sketch proof of the following, which may be the most elementary result in this line.

\begin{prop}
For any locally compact, second countable group $G$ and Polish $G$-module $A$, any class in $\rmH_\m^p(G,A)$ has a representative $G^p\to A$ that is continuous on a G$_\delta$ subset of $G^p$ of full Haar measure (which is therefore also dense).
\end{prop}

This is clearly true for $p=1$ because crossed homomorphisms are automatically continuous.  Na\"\i vely, it would then follow for all $p$ by dimension-shifting if one could show that given a continuous surjection of Polish $G$-modules $A\onto B$ and any measurable cochain $\s:G^p\to B$ satisfying the conclusions of the proposition, $\s$ can be lifted to a cochain $\s':G^p\to A$ still satisfying those conclusions.

Unfortunately it is not clear how to prove this directly, but this approach can be repaired by refining the desired conclusions a little.  The key is to impose some extra requirements on the possible structure of the set of discontinuities of the cochain.  A suitable refinement is to prove instead that:
\begin{quote}
Any class in $\rmH_\m^p(G,A)$ has a representative $G^p\to A$ that is a uniform limit of functions $G^p \to A$ which are locally constant outside some Haar-negligible closed subset of $G^p$.
\end{quote}
Since any such uniform limit is still continuous outside the union of those negligible closed subsets, we obtain the desired full-measure G$_\delta$ set of continuity.  On the other hand, it is fairly simple to prove that
\begin{itemize}
\item any continuous function $G^p\to A$ may be uniformly approximated by a.e. locally constant functions of the kind described above, and
\item if $A \onto B$ is a continuous surjection of $G$-modules then any function $G^p\to B$ which is a uniform limit of the kind introduced above may be lifted to a function $G^p\to A$ which has the same structure (simply by lifting each of the approximants in a compatible way).
\end{itemize}

From these facts, the complete proof by dimension-shifting follows quickly.

In order to recover Proposition~\ref{prop:semi-alg-rep} instead in case $G$ is a Lie group, one simply makes a further refinement to the above formulation of the desired structure of cocycles to prove that:

\begin{quote}
Any class in $\rmH_\m^p(G,A)$ has a representative $G^p\to A$ that is a uniform limit of functions $G^p \to A$ which are locally constant outside some closed nowhere-dense semi-algebraic subset of $G^p$.
\end{quote}

Once again, the facts that all continuous functions may be so approximated and that functions of this structure can be lifted through $G$-module quotients are now easy exercises.  In case $A$ is discrete, any uniformly convergent sequence of such semi-algebraic representatives must eventually stabilize, giving the conclusion.

\section{Recap of open questions}

We close by recollecting the various open questions posed earlier in the paper.

\subsection{Questions on the structure of Polish modules}

\begin{itemize}
\item Question~\ref{ques:R-in-L0}: Consider $\bbR$ identified with the subspace of constant functions in $L^0([0,1])$, the space of a.e.-equivalence classes of measurable functions $[0,1]\to \bbR$.  Does it have a continuous cross-section for the topology of convergence in probability?

\item Question~\ref{ques:subgp-of-L1(sq)}: Does the closed, contractible subgroup
\[\pi_1^\ast L^1([0,1],\bbZ) + \pi_2^\ast L^1([0,1],\bbZ)\]
admit a local cross-section in $L^1([0,1]^2,\bbZ)$, where $\pi_i:G\times G\to G$ for $i=1,2$ are the coordinate projections?
\end{itemize}

\subsection{Questions on analytic properties of cohomology groups}

\begin{itemize}
\item Question~\ref{ques:thm-B-fail}: Are there examples of locally compact, second countable, compactly-generated groups $G$ for which Theorem B fails in degree $2$ for toral or discrete targets?

\item Question~\ref{ques:B-or-C-with-extra-assumptions}: Does Theorem B or C hold if $G$ is locally compact, compactly-generated and second countable, $G_n = G/K_n$ for some decreasing sequence of compact subgroups $K_n \unlhd G$, and if
\begin{itemize}
\item $G$ is connected, or
\item the first quotient $G/K_1$ is a discrete $\rm{FP}_\infty$ group (see Chapter VIII of Brown~\cite{Bro82})?
\end{itemize}

\item Question~\ref{ques:topologizing-cohom-for-disc-G}: Is there a countable discrete group $\G$ for which $\B^p(\G,\bbZ)$ is not closed in $\Z^p(\G,\bbZ)$ for the topology of pointwise convergence?  Can $\G$ be finitely-generated, or even finitely-presented?  (It is easy to see that this would need $p\geq 2$.)
\end{itemize}

\subsection{Questions on vanishing or agreement between cohomology theories}

\begin{itemize}
\item Question~\ref{ques:bad-compact-and-contractible}: Is there a non-locally-convex F-space $A$, or other contractible Polish group, which admits a continuous action of a compact metric group $G$ such that $\rmH_\m^p(G,A) \neq (0)$ for some $p \geq 1$?

\item Question~\ref{ques:contractible-disagree-with-ss}: Are there a connected locally compact group $G$ and a non-locally-convex F-space $G$-module $A$, or more generally a contractible Polish module, for which
\[\rmH_\m^p(G,A) \not\cong \rmH^p_\rm{ss}(G,A)\]
for some $p$?

\item Question~\ref{ques:elem-compar-contractible}: Is there a direct proof that $\rmH_\m^p(G,A) = (0)$ whenever $G$ is contractible, $A$ is discrete and $p \geq 1$?

\item Question~\ref{ques:elem-compar-for-discrete}: Is there an elementary proof that for any locally compact, second countable $G$ and a discrete $G$-module $A$ with trivial action one has
\[\rmH_\m^\ast(G,A) \cong \rmH^\ast_\rm{m.sing}(B_G,A) \cong \rmH^\ast_\rm{sheaf}(B_G,\cal{A})?\]
\end{itemize}

\section*{Acknowledgements}

Our thanks go to Matthias Flach, Karl Hofmann, Arati Khairnar, Stephen Lichtenbaum, Nicolas Monod, C.S. Rajan and Terence Tao for several helpful communications. \fin

\bibliographystyle{abbrv}
\bibliography{Cty_of_meas_cohom}

\end{document}